\documentclass{amsart}

\usepackage{amsmath}
\usepackage{amssymb,amsfonts,amsthm}
\usepackage{cite}
\usepackage{graphicx}
\usepackage[colorlinks=true, allcolors=blue]{hyperref}
\usepackage[all]{xy}
\usepackage{color,xcolor}

\usepackage{mathrsfs}
\usepackage[all]{xy}
\usepackage{mathabx} 
\usepackage{mathtools}
\usepackage{mathbbol}

\usepackage{wasysym}
\usepackage{cancel,ulem}

\usepackage{mathtools}
\usepackage{mathbbol}
\usepackage{tikz}

\usepackage{ifthen}

\newtheorem{theorem}{Theorem}[section]
\newtheorem{lemma}[theorem]{Lemma}
\newtheorem{proposition}[theorem]{Proposition}
\newtheorem{corollary}[theorem]{Corollary}

\theoremstyle{definition}
\newtheorem{defn}[theorem]{Definition}

\theoremstyle{remark}
\newtheorem{remark}[theorem]{Remark}
\newtheorem{example}[theorem]{Example}
\newtheorem*{claim}{Claim}
\newtheorem{fact}{Fact}
\newtheorem{question}{Question}

\def\Z{\mathbb Z}
\def\R{\mathbb R}
\def\F{\mathbb F}
\def\N{\mathbb N}
\def\H{\mathcal H}

\def\G{\mathcal G}
\def\C{\mathcal C}

\def\Isom{\mathrm{Isom}}

\def\rk{\mathrm{rk}}

\def\out{\mathrm{Out}}

\def\d{\mathrm{dist}}

\def\pmcg{\mathrm{PMCG}}
\def\sl{\mathrm{SL}}

\title{Quasi-retracts of groups}

\author{Renxing Wan}
\address{School of Mathematical Sciences,  Key Laboratory of MEA (Ministry of Education) \& Shanghai Key Laboratory of PMMP,  East China Normal University, Shanghai 200241, China P. R.}
\email{rxwan@math.ecnu.edu.cn}

\begin{document}

\begin{abstract}
    In this paper, we study a special class of quasi-homomorphisms, i.e. quasi-retractions from a group to its subgroups. We first give some algebraic and geometric properties of quasi-retracts and then propose a theory of quasi-split short exact sequences of groups. Later, we establish a connection between quasi-homomorphisms and induced quasi-actions. Finally, we give some geometric applications of quasi-homomorphisms, including normal quasi-retracts inherit cobounded actions on hyperbolic spaces, properties (QFA), (QT') and (PH') are all stable under left quasi-split group extensions, quasi-isomorphic groups have isomorphic hyperbolic structures, and so on.
\end{abstract}
\maketitle
%\tableofcontents

\section{Introduction}

\subsection{Quasi-homomorphisms and quasi-retracts}

Let $G$ be a group and $H$ be a group equipped with a proper left-invariant metric $d$ (e.g., a finitely generated group, equipped with a word metric). A map $\phi : G \to H$ is called a \textit{quasi-homomorphism} if there exists a constant $C>0$ such that $$d(\phi(xy), \phi(x)\phi(y)) \le C$$ for all $x, y \in G$. For two maps $\phi,\phi': G\to H$, we define $$\d(\phi,\phi'):=\sup_{x\in G}d(\phi(x),\phi'(x)).$$ Two quasi-homomorphisms $\phi$ and $\phi'$ are called \textit{equivalent} if $\d(\phi,\phi')<\infty$. In the case when $H$ is discrete (and in this paper we limit ourselves only to this class of groups), $\phi$ is a quasi-homomorphism if and only if the set of defects of $\phi$ $$D(\phi)=\{\phi(y)^{-1}\phi(x)^{-1}\phi(xy): x,y\in G\}$$ is finite. And two quasi-homomorphisms $\phi, \phi'$ are equivalent if and only if $\{\phi(x)^{-1}\phi'(x): x\in G\}$ is finite. A quasi-homomorphism with values in $\Z$ (or $\R$, equipped with the standard metric) is called a \textit{quasimorphism}. 

The concept of quasi-homomorphisms goes back to Ulam \cite[Chapter 6]{Ula60}, who asked if they are close to group homomorphisms. Later, Fujiwara-Kapovich \cite{FK16} studied this concept to figure out why it is so ``hard to construct'' quasi-homomorphisms to noncommutative groups which are neither homomorphisms, nor come  from quasi-homomorphisms with commutative targets, provided that $H$ is a discrete group. They showed that every quasi-homomorphism is constructible \cite[Theorem 1.2]{FK16}. 

In this paper, we study a special class of quasi-homomorphisms, i.e. quasi-homomorphisms from a group to its subgroups.

For convenience, we say $(G,H)$ is a \textit{group pair} if $G$ is a discrete group and $H\le G$ is a subgroup. Let $(G,H)$ be a group pair. Classically, $H$ is called a \textit{retract} of $G$ if there exists a homomorphism $r: G\to H$ such that $r$ is left-inverse to the inclusion map $\iota: H\to G$. This homomorphism $r$ is called a \textit{retraction} from $G$ to $H$. The study of retractions have been lasting for decades in Geometric Group Theory and relates to many other topics. For example, \cite{STZ22, MV03, Tur96} studied retracts of free groups to characterize fixed subgroups of automorphisms of free groups and \cite{Min21, LR08} introduced the notion of virtual retractions to study the geometry and topology of hyperbolic manifolds and discrete groups.

In this paper, we define a subgroup $H$ of $G$ to be a \textit{quasi-retract}\footnote{This word ``quasi-retraction'' was first used by Alonso in \cite{Alo94} to denote a certain subordinate relationship between metric spaces.} if there exists a quasi-homomorphism $r: G\to H$ such that $r\circ \iota$ is equivalent to $Id_H$ where $\iota: H\to G$ is the inclusion map. This quasi-homomorphism $r$ is called a \textit{quasi-retraction} from $G$ to $H$. By definition, retracts are quasi-retracts.

A relationship between quasimorphisms and quasi-retracts is that a group admits an unbounded quasimorphism if and only if it contains an infinite cyclic quasi-retract (cf. Lemma \ref{Lem: UnbddQM}).

Another thing worth mentioning is that in \cite{HS16}, Hartnick-Schiweitzer defined a map $\phi: G\to H$ to be a quasi-homomorphism if and only if any quasimorphism on $H$ can be pull-backed to a quasimorphism on $G$ by composing $\phi$. To make a distinction, we shall call such a quasi-homomorphism an \textit{HS-quasi-homomorphism}. Quasi-homomorphisms and HS-quasi-homomorphisms are not equivalent in general; see \cite[Subsection 9.3]{FK16} for a detailed discussion about different quasi-homomorphisms. For a group pair $(G,H)$, Hartnick-Schiweitzer also defined an HS-quasi-homomorphism $r: G\to H$ to be an \textit{HS-quasi-retraction} if the composition $r\circ \iota$ is equivalent to $Id_H$ where $\iota: H\to G$ is the inclusion map.

\subsection{Extension problem}

In \cite{HO13}, Hull-Osin considered the following ``extension problem'': 
\begin{question}
    Under what conditions can a 1-quasi-cocycle (eg. a quasimorphism) on a subgroup be extended to the whole group?
\end{question}
They solved the problem by assuming the subgroup is hyperbolically embedded. Roughly speaking, hyperbolically embedded subgroups generalize the concept of parabolic subgroups in relatively hyperbolic groups; see \cite{DGO17} for details about hyperbolically embedded subgroups. Moreover, under the same assumption, Frigerio-Pozzetti-Sisto \cite{FPS15} solved a ``higher-dimensional extension problem''.

From the definition, any quasimorphism on a quasi-retract can be extended via this quasi-retraction to a quasimorphism on the whole group.  So quasi-retracts naturally solve the ``extension problem'' about quasimorphisms. 

The first result of this paper is to explore the relationship between hyperbolically embedded subgroups and quasi-retracts. We denote by $\F_n$ the free group of rank $n$. A quasi-retract is called \textit{nontrivial} if it is neither a finite subgroup nor the whole group. 
\begin{theorem}\label{IntroThm2}
    \begin{enumerate}
        \item\label{ite1} A finite-by-$\Z^m$ hyperbolically embedded subgroup is always a quasi-retract. 
        \item\label{ite2} A subgroup of $\F_2$ is a nontrivial quasi-retract if and only if it is a cyclic group. Moreover, there exists a hyperbolically embedded subgroup of $\F_2$ which is not a quasi-retract.
    \end{enumerate}
\end{theorem}

\begin{remark}
    \begin{enumerate}
        \item It is well known and easy to prove that every infinite virtually cyclic group is either finite-by-(infinite cyclic) or finite-by-(infinite dihedral). Hence, for a group without involutions, any virtually cyclic subgroup is finite-by-cyclic.
        \item There are several typical examples of groups containing virtually cyclic hyperbolically embedded subgroups (cf. \cite{HO13}):
        \begin{itemize}
            \item $G$ is the mapping class group of a punctured closed orientable surface and $H\le G$ is the maximal virtually cyclic subgroup containing a pseudo-Anosov element;
            \item $G=\out(\F_n)$ and $H\le G$ is the maximal virtually cyclic subgroup containing a fully irreducible automorphism;
            \item $G$ is a group acting properly on a proper CAT(0) space and $H\le G$ is the maximal virtually cyclic subgroup containing a rank-1 element.
        \end{itemize}
    \end{enumerate}
\end{remark}

\subsection{Quasi-split short exact sequence}

Let 
\begin{equation}\label{Equ: SES}
    1\to H\xrightarrow{\iota} G\xrightarrow{\pi} Q\to 1 \tag{*}
\end{equation}
be a short exact sequence of groups. Classically, the short exact sequence (\ref{Equ: SES}) is \textit{left split} if there exists a retraction from $G$ to $H$. It is \textit{right split} if there exists a section $s: Q\to G$ such that $s$ is a homomorphism. Similarly, we say the short exact sequence (\ref{Equ: SES}) is \textit{left quasi-split} if there exists a quasi-retraction from $G$ to $H$. It is \textit{right quasi-split} if there exists a section $s: Q\to G$ such that $s$ is a quasi-homomorphism. We remark that the notion of right quasi-split here is the same as the notion of \textit{quasisplit} in \cite[Subsection 2.4.1]{FK16}.

\begin{theorem}[Classical version]\label{IntroThm: SSES}
    Let $1\to H\xrightarrow{\iota} G\xrightarrow{\pi} Q\to 1$ be a short exact sequence of groups. The following are equivalent: 
    \begin{enumerate}
        \item The short exact sequence is left split. Or equivalently, $H$ is a retract of $G$.
        \item There exists a section $s: Q\to G$ such that $s$ is a homomorphism and $H$ commutes with $s(Q)$.
        \item There is an isomorphism $\phi: G\to H\times Q$ such that the diagram commutes: 
        $$\xymatrix{
  1  \ar[r]^{} & H \ar[d]_{Id} \ar[r]^{\iota} & G \ar[d]_{\phi} \ar[r]^{\pi} & Q \ar[d]_{Id} \ar[r]^{} & 1  \\
  1 \ar[r]^{} & H \ar[r]^{} & H\times Q \ar[r]^{} & Q \ar[r]^{} & 1   }$$
  where the bottom row is the short exact sequence for a direct product.
    \end{enumerate}
\end{theorem}

A section $s$ is \textit{normalized} if $s(1)=1$. For two subsets $K,T$ of a group $G$, we define $\C(K,T)=\{[k,t]=ktk^{-1}t^{-1}: k\in K, t\in T\}$. In this case, $K$ and $T$ are commutative (resp. \textit{almost commutative}) if and only if $\C(K,T)=\{1\}$ (resp. $\C(K,T)$ is finite). 

For two discrete groups $G,G'$, a quasi-homomorphism $\phi: G\to G'$ is a \textit{quasi-isomorphism} if it admits a \textit{quasi-inverse}, i.e., a quasi-homomorphism $\phi': G'\to G$ such that $\d(\phi'\circ \phi, Id_G)<\infty$ and $\d(\phi\circ \phi', Id_{G'})<\infty$. A quasi-isomorphism is \textit{strict} if $\phi'=\phi^{-1}$. The second result of this paper is an analogue of Theorem \ref{IntroThm: SSES}.

\begin{theorem}[Theorem \ref{Thm: LQS}]\label{IntroThm: QSSES}
    Let $1\to H\xrightarrow{\iota} G\xrightarrow{\pi} Q\to 1$ be a short exact sequence of groups.  The following are equivalent: 
    \begin{enumerate}
        \item The short exact sequence is left quasi-split. Or equivalently, $H$ is a quasi-retract of $G$.
        \item There exists a normalized section $s: Q\to G$ such that $s$ is a quasi-homomorphism and $H$ almost commutes with $s(Q)$.
        \item There is a strict quasi-isomorphism $\phi: G\to H\times Q$ such that the diagram commutes: 
        $$\xymatrix{
  1  \ar[r]^{} & H \ar[d]_{Id} \ar[r]^{\iota} & G \ar[d]_{\phi} \ar[r]^{\pi} & Q \ar[d]_{Id} \ar[r]^{} & 1  \\
  1 \ar[r]^{} & H \ar[r]^{} & H\times Q \ar[r]^{} & Q \ar[r]^{} & 1   }$$
  where the bottom row is the short exact sequence for a direct product.
    \end{enumerate}
\end{theorem}

\begin{remark}
    The direction ``$(2)\Rightarrow (3)$'' in Theorem \ref{IntroThm: QSSES} is also an analogue of \cite[Proposition A.2]{Alo94} which states that by adding some finiteness conditions to the group extension (\ref{Equ: SES}), $G$ is quasi-isometric to $H\times Q$. We will see later that quasi-isomorphisms are quasi-isometries for finitely generated groups.
\end{remark}

\paragraph{\textbf{Central extensions}} An important class of short exact sequences comes from central extensions. Let 
\begin{equation}\label{Equ: CenExt}
    1\to Z\xrightarrow{\iota}G\xrightarrow{\pi}Q\to 1 \tag{**}
\end{equation}
be a central extension of groups, where $Q$ and $Z$ (hence, $G$) are finitely generated. Any such extension defines a cohomology class $\omega\in H^2(Q,Z)$ which will be called the \textit{Euler class} of the extension. It is well known that the Euler class completely determines the isomorphism class of a central extension, and it is natural to investigate which geometric features it encodes.

We say that a class $\omega\in H^2(Q,Z)$ is \textit{bounded} if it lies in the image of the comparison map $H^2_b(Q,Z)\to H^2(Q,Z)$, i.e. if it can be described by a bounded cocycle.
In \cite{FK16}, Fujiwara-Kapovich proved that the central extension (\ref{Equ: CenExt}) has a bounded Euler class if and only if it is right quasi-split. By Theorem \ref{IntroThm: QSSES}, this is also equivalent to saying that $Z$ is a quasi-retract of $G$. For more equivalent conditions for (\ref{Equ: CenExt}) to have a bounded Euler class, we refer the readers to \cite[Lemma 5.6]{TW25}.

Following \cite{Ger92}, we say that the extension of finitely generated groups
$$1\to Z\xrightarrow{\iota}G\xrightarrow{\pi}Q\to 1$$
is \textit{quasi-isometrically trivial} if there exists a quasi-isometry $\phi : G \to Z \times Q$
such that the following diagram commutes, up to bounded error:
$$\xymatrix{
  G \ar[d]_{\phi} \ar[r]^{\pi} & Q \ar[d]_{Id} \\
  Z\times Q \ar[r]^{\pi_2} & Q }$$
Here $\pi_2 : Z \times Q \to Q$ is the projection on the second factor, and, for clarity, the diagram commuting up to bounded error means that given a word metric $d_Q$ on $Q$ we have $\sup_{g\in G} d_Q(\pi(g), \pi_2(\phi(g))) < +\infty$.

Gersten proved in \cite[\textsection 3]{Ger92} that if the Euler class of the central extension is bounded, then the extension is quasi-isometrically trivial. As for the converse direction, Frigerio-Sisto investigated in \cite{FrS23} when a quasi-isometrically trivial central extension has a bounded Euler class. They proved that a central extension is quasi-isometrically trivial if and only if it has a weakly bounded Euler class. As an analogue, Theorem \ref{IntroThm: QSSES} implies that 
\begin{corollary}
    The central extension (\ref{Equ: CenExt}) has a bounded Euler class if and only if it is it is ``strictly quasi-isomorphically trivial'', i.e. there is a strict quasi-isomorphism $\phi: G\to Z\times Q$ such that the diagram commutes: 
        $$\xymatrix{
  1  \ar[r]^{} & Z \ar[d]_{Id} \ar[r]^{\iota} & G \ar[d]_{\phi} \ar[r]^{\pi} & Q \ar[d]_{Id} \ar[r]^{} & 1  \\
  1 \ar[r]^{} & Z \ar[r]^{} & Z\times Q \ar[r]^{} & Q \ar[r]^{} & 1   }$$
  where the bottom row is the short exact sequence for a direct product.
\end{corollary}
Since bounded Euler classes imply weakly bounded Euler classes, we will see in Lemma \ref{Lem: QIToQI} that quasi-isomorphisms imply quasi-isometries for finitely generated groups.

%we refer the readers to \cite{AM24, FrS23} for counterexamples.

The following example presents an important class of normal quasi-retracts. For the definition of amalgamated direct products, we refer the readers to Definition \ref{Def: AmalDirPro}.

\begin{example}
    Let $\Sigma$ be a closed surface with finitely many marked points. Let $C$ be a separating essential simple closed curves in $\Sigma$. Let $\pmcg(\Sigma)$ be the pure mapping class group of $\Sigma$ and $\pmcg(\Sigma,C)$ denote the stabilizer in $\pmcg(\Sigma)$ of $C$. Since $C$ separates $\Sigma$, $C$ divides $\Sigma$ into two subsurfaces $\Sigma_-, \Sigma_+$ with a common boundary component $C$. It is well-known and easy to prove that $\pmcg(\Sigma,C)$ is an amalgamated direct product of $\pmcg(\Sigma_-)$ and $\pmcg(\Sigma_+)$. Suppose that $\Sigma_+'$ is the surface obtained from $\Sigma_+$ by capping the boundary component $C$ with a once-marked disk. By \cite[Proposition 6.4]{FK16}, the following central extension $$1\to \langle T_C\rangle \to \pmcg(\Sigma_+)\xrightarrow{Cap} \pmcg(\Sigma_+') \to 1$$  has a bounded Euler class where $T_C$ is the Dehn twist around $C$. As a result of Lemma \ref{Lem: ExaOfQR} (\ref{5.4.3}), $\pmcg(\Sigma_-)$ is a normal quasi-retract of $\pmcg(\Sigma, C)$.
\end{example}

\subsection{Induced quasi-actions}

In \cite{Man06}, Manning defined group quasi-actions as a natural coarse generalization of isometric group actions. Specifically, a \textit{$(\lambda, \epsilon)$-quasi-action} of a group $G$ on a metric space $X$ is a map $\rho: G \times X \to X$, denoted $\rho(g, x) \mapsto gx$, so that the following hold:
    \begin{itemize}
        \item[(i)] for each $g$, $\rho(g, -): X \to X$ is a $(\lambda, \epsilon)$-quasi-isometry;
        \item[(ii)] for each $x \in X$ and $g, h \in G$, we have $$d(g(hx),(gh)x) = d(\rho(g, \rho(h, x)), \rho(gh, x)) \le \epsilon.$$
    \end{itemize}
(Note that $\lambda$ and $\epsilon$ must be independent of $g$ and $h$.) We call a quasi-action
\textit{cobounded} if, for every $x \in X$, the map $\rho(-, x): G \to X$ is $\epsilon'$-coarsely surjective for some $\epsilon'\ge 0$.

An important geometric application of quasi-homomorphisms is that they can induce new quasi-actions on metric spaces. A key lemma is as follows.

\iffalse
Let $\phi: G\to H$ be a quasi-homomorphism with the defect set $D=D(\phi)$. Suppose $H$ admits a $(\lambda,\epsilon)$-quasi-action on a metric space $X$ and there exists a constant $M>0$ such that $$\sup_{x\in X, a\in D}d(x,ax)\le M.$$ Define a map $\rho: G\times X\to X$ by $\rho(g,x):=\phi(g)x$. For any $g,h\in G$, there exists $a\in D$ such that $\phi(gh)=\phi(g)\phi(h)a$. Thus, for each $x\in X$ and $g,h\in G$, we have 
\begin{align*}
    d(\rho(g,\rho(h,x)),\rho(gh,x))&=d(\phi(g)(\phi(h)x), \phi(gh)x)\le d((\phi(g)\phi(h))x,(\phi(g)\phi(h)a)x)+\epsilon\\ &\le d((\phi(g)\phi(h))x,(\phi(g)\phi(h))(ax))+2\epsilon\le \lambda d(x,ax)+3\epsilon\le \lambda M+3\epsilon.
\end{align*}
The above inequality shows that $\rho$ is a $(\lambda,\lambda M+3\epsilon)$-quasi-action of $G$ on $X$. This quasi-action is referred to as \textit{an induced quasi-action by $\phi$}. Furthermore, if $\phi$ is a quasi-retraction, then the induced quasi-action of $G$ on $X$ is an extension of the original quasi-action of $H$ on $X$. 
\fi

\begin{lemma}[Lemma \ref{Lem: QHToQA}]\label{IntroLem: QHToQA}
    Let $\phi: G\to H$ be a coarsely surjective quasi-homomorphism. Suppose $H$ admits a cobounded quasi-action on a metric space $X$. Then the map $\rho: G\times X\to X$ defined by $\rho(g,x):=\phi(g)x$ is an induced quasi-action of $G$ on $X$ by $\phi$.

    Moreover, suppose $H$ admits another cobounded quasi-action on a metric space $Y$. Then the quasi-action of $H$ on $X$ is quasi-conjugate to the quasi-action of $H$ on $Y$ if and only if the induced quasi-action of $G$ on $X$ by $\phi$ is quasi-conjugate to the induced quasi-action of $G$ on $Y$ by $\phi$.
\end{lemma}

\paragraph{\textbf{Property (QFA), (QT) and (PH)}} To apply Lemma \ref{IntroLem: QHToQA}, we first introduce some geometric properties of groups acting isometrically on hyperbolic spaces.

In \cite{Man06}, Manning defined a group $G$ to have \textit{property (QFA)} if for every quasi-action of $G$ on any tree $X$, there is some $x\in X$ so that the orbit $Gx$ has finite diameter. This property is analogous to Serre's property (FA). A typical class of groups with property (QFA) are $\sl_n(\Z)$ for $n\ge 3$ \cite{Man06}.

In \cite{BBF21}, Bestvina-Bromberg-Fujiwara defined a finitely generated group $G$ to have \textit{property (QT)} if $G$ acts isometrically on a finite product of quasi-trees equipped with $\ell^1$-metric such that the orbit map is a quasi-isometric embedding. They also showed that residually finite hyperbolic groups and mapping class groups have property (QT). Following \cite{Tao24}, a finitely generated group $G$ has \textit{property (QT')} if it has property (QT) and the action on a finite product of quasi-trees is induced by a diagonal action. By a result of Button \cite{But22}, a finitely generated group $G$ with property (QT) virtually has property (QT'), i.e. $G$ contains a finite-index subgroup $G_0$ which has property (QT'). 

In \cite{TW25}, Tao and the author defined a group $G$ to have \textit{property (PH')} if there exist finitely many hyperbolic spaces $X_1,\ldots, X_n$ on which $G$ acts coboundedly and the induced diagonal action of $G$ on the product $\prod_{i=1}^nX_i$ with $\ell^1$-metric is proper. A group $G$ has \textit{property (PH)} if $G$ virtually has property (PH'). Some typical examples of groups with property (PH) include finitely generated abelian groups, Coxeter groups, virtually colorable hierarchically hyperbolic groups and most 3-manifold groups. As shown in \cite[Lemma 2.2]{TW25}, if a finitely generated group has property (QT) or (QT'), then it has property (PH) or (PH'), respectively. We refer the readers to \cite[Figure 1]{TW25} for an illustration of relations between these properties.

By combining Lemma \ref{IntroLem: QHToQA} and a result of Manning (cf. Proposition \ref{Prop: QAToA}), we can obtain some quasi-isomorphism invariants for discrete groups. Since some of them are not quasi-isometry invariants, we know that quasi-isomorphisms are  strictly stronger equivalence relations than  quasi-isometries for finitely generated groups.

\begin{proposition}\label{IntroProp: QIInv}
    \begin{enumerate}
        \item Almost commutativity is a quasi-isomorphism invariant for discrete groups.
        \item Property (QFA) is a quasi-isomorphism invariant for finitely generated groups.
        \item Property (PH’) is a quasi-isomorphism invariant for discrete groups.
        \item Property (QT’) is a quasi-isomorphism invariant for finitely generated groups.
    \end{enumerate}
\end{proposition}

In \cite{TW25}, Tao and the author proved the stability of properties (QT), (QT'), (PH) and (PH') under left quasi-split central extensions. Here, we give a partial generalization of \cite[Theorems 1.7, 1.8]{TW25}. A group is called \textit{FZ} if its center is of finite index.

\begin{theorem}\label{IntroThm: PH&QTStability}
    Let $1\to H\to G\to Q\to 1$ be a short exact sequence of finitely generated groups which is left quasi-split. Then 
    \begin{enumerate}
        \item $G$ has property (QFA) if and only if  both $H$ and $Q$ have property (QFA).
        \item $G$ has property (PH') or (QT') if and only if  both $H$ and $Q$ have property (PH') or (QT'), respectively.
        \item suppose that $H$ is FZ. Then $G$ has property (PH) or (QT) if and only if  $Q$ has property (PH) or (QT), respectively.
    \end{enumerate}
\end{theorem}

\paragraph{\textbf{Hyperbolic structures}} For every group $G$, Abbott-Balasubramanya-Osin \cite{ABO19} defined the set of \textit{hyperbolic structures} on $G$, denoted by $\H(G)$, which consists of equivalence classes of (usually infinite) generating sets of $G$ such that the corresponding Cayley graph is hyperbolic; two generating sets of $G$ are \textit{equivalent} if the corresponding word metrics on $G$ are bi-Lipschitz equivalent. See \textsection \ref{subsec: HypSpace} or \cite{ABO19} for more details about hyperbolic structures on groups.

Finally, we obtain a relationship on hyperbolic structures between a group and its quasi-retracts.

\begin{proposition}\label{IntroProp: HypStr}
    \begin{enumerate}
        \item Let $\phi: G\to H$ be a coarsely surjective quasi-homomorphism. Then $\H(H)$ embeds in $\H(G)$ as a sub-poset.
        \item Let $G$ and $G'$ be two discrete quasi-isomorphic groups. Then $\H(G)\cong \H(G')$.
    \end{enumerate}
\end{proposition}

\subsection*{Structure of the paper}
The paper is organized as follows. Section \ref{sec: Preliminary} is devoted to defining some notations and studying some elementary properties of quasi-homomorphisms and quasi-isomorphisms. In Section \ref{sec: quasi-retraction}, we first study some basic properties of quasi-retractions and then characterize the quasi-retracts of torsion-free hyperbolic groups and some more general groups. In Section \ref{sec: QSSES}, we introduce the classical theory of split short exact sequences and prove Theorem \ref{IntroThm: QSSES}. By utilizing Theorem \ref{IntroThm: QSSES}, we complete the proof of Theorem \ref{IntroThm2}. In Section \ref{sec: IndQA}, we recall some basic material about group actions on hyperbolic spaces, and then prove Proposition \ref{IntroProp: QIInv}, Theorem \ref{IntroThm: PH&QTStability} and Proposition \ref{IntroProp: HypStr} in turn.

\subsection*{Acknowledgments}
We are grateful to Zhenguo Huangfu and Bingxue Tao for many helpful suggestions on the first draft. R. W. is supported by NSFC No.12471065 \& 12326601 and in part by Science and Technology Commission of Shanghai Municipality (No. 22DZ2229014).

\section{Preliminaries}\label{sec: Preliminary}

\subsection{Definitions and notations}
Let $H$ be a group with a proper left-invariant metric $d$. For a subset $A\subset H$ and a constant $C>0$, we denote $\mathcal N_C(A):=\{h\in H: d(h,A)\le C\}$.

Let $S\subset H$ be a finite subset. For convenience, we will always write  $$h\sim_{S}h'$$ if $h = h's$ with $h, h' \in H, s \in S$. For example, for a quasi-homomorphism $\phi : G \to H$ with a defect set $D = D(\phi)$, by the definition, $$\phi(ab) \sim_D \phi(a)\phi(b) \text{ and } \phi(abc)\sim_D\phi(a)\phi(bc)\sim_D\phi(a)\phi(b)\phi(c) $$ for $a, b,c \in G$. 

\iffalse
\begin{defn}\cite[Definition 2.1]{FK16}\label{Def: AlmostHomo}
    Suppose that a map $\phi : G \to H$ between groups has the property that $\phi(G)$ is contained in a subgroup $J<H$, $J$ contains a finite normal subgroup $K\lhd J$, such that the projection $\bar \phi: G\to \bar J=J/K$ is a homomorphism. We then will refer to $\phi$ as an \textit{almost homomorphism}, it is a homomorphism modulo a finite normal subgroup (in the range of $\phi$).
\end{defn}

Clearly, every almost homomorphism is a quasi-homomorphism.
\fi

For a subset $D$ of a group $H$ and $n \ge 2$ we will use the notation $D^n$ to denote
the subset of $H$ consisting of products of at most $n$ elements of $D$. More generally,
for two subsets $A, B \subset H$ we let
$$A \cdot B = \{ab : a \in A, b \in B\}.$$
We will use the notation $D^{-1}$ for the set of inverses of elements of $D$. Then $$h\sim_D h'\Longleftrightarrow h'\sim_{D^{-1}}h.$$

Given a group $H$ and its subgroup $A$ we let $N_H(A)$ and $Z_H(A)$ denote the normalizer and the centralizer of $A$ in $H$ respectively. %We also denote $\langle\langle A\rangle\rangle$ as the normal closure of $A$, i.e., the smallest normal subgroup containing $A$  in $H$.

Given a group $G$ with a finite generating set $S$, we denote by $d_S$ the word metric on $G$.

Given a quasi-homomorphism $\phi: G\to H$, we denote by $\Delta_{\phi}$ the \textit{defect subgroup} of $H$ which is finitely generated by $D(\phi)$.

\subsection{Elementary properties of quasi-homomorphisms}

We collect some basic facts about quasi-homomorphisms. See \cite[Subsections 2.2, 2.3, 3.1]{FK16} for their proofs.
\begin{lemma}\label{Lem: QHProperties}
    Let $\phi:G\to H$ be a quasi-homomorphism with $D=D(\phi)$. 
    \begin{enumerate}
        \item\label{def} For any $x_1,\cdots,x_n\in G$, $\phi(x_1\cdots x_n)\sim_{D^{n-1}}\phi(x_1)\cdots \phi(x_n)$.
        \item\label{inverse} For any $x\in G$, $\phi(x)^{-1}\sim_{D^2}\phi(x^{-1})$.
        \item\label{conjugate} For any $h\in \phi(G)$, $h^{-1}Dh\subset D^2D^{-1}$.
        \item\label{normalizer} $\phi(G)\cup \phi(G)^{-1}\subset N_H(\Delta_{\phi})$ and the induced map $\bar \phi: G\to N_H(\Delta_{\phi})/\Delta_{\phi}$ is a homomorphism.
        \item\label{centralizer} There exists a constant $C$ such that $\phi(G)\subset \mathcal N_C(Z_H(\Delta_{\phi}))$.
    \end{enumerate}
\end{lemma}

As a corollary, we have
\begin{lemma}\label{Lem: SujQHImpEle}
    Let $\phi: G\to H$ be a surjective quasi-homomorphism. Then the defect subgroup $\Delta_{\phi}$ is a virtually abelian normal subgroup of $H$. %In particular, if $r: G\to H$ is a quasi-retraction, then $\Delta_{\phi}$ is a virtually abelian normal subgroup of $H$.
\end{lemma}
\begin{proof}
    Since $\phi$ is surjective, Lemma \ref{Lem: QHProperties} (\ref{normalizer}) shows that  $H=\phi(G)=N_H(\Delta_{\phi})$. Thus $\Delta_{\phi}\lhd H$. By Lemma \ref{Lem: QHProperties} (\ref{centralizer}), there exists a constant $C$ such that $\phi(G)=H\subset \mathcal N_C(Z_H(\Delta_{\phi}))$. This implies that $[H: Z_H(\Delta_{\phi})]<\infty$. Denote $\Delta_{\phi}'=\Delta_{\phi}\cap Z_H(\Delta_{\phi})$. Note that $\Delta_{\phi}'$ is an abelian group. According to the isomorphism theorem of groups, $[\Delta_{\phi}: \Delta_{\phi}']\le [H: Z_H(\Delta_{\phi})]<\infty$, which implies that $\Delta_{\phi}$ is virtually abelian.
\end{proof}

\begin{lemma}\label{Lem: inverse}
    Let $\phi:G\to H$ be a quasi-homomorphism with $D=D(\phi)$. For any $x,y\in G$,  $\phi(x^{-1}y)\sim_D\phi(x^{-1})\phi(y)\sim_{(D^2D^{-1})^{-2}}\phi(x)^{-1}\phi(y)$.
\end{lemma}
\begin{proof}
    By Lemma \ref{Lem: QHProperties} (\ref{inverse}), there exist $a_1,a_2\in D^{-1}$ such that $\phi(x^{-1})=\phi(x)^{-1}a_1a_2$. Then by Lemma \ref{Lem: QHProperties} (\ref{conjugate}), one has $$\phi(x^{-1})\phi(y)=\phi(x)^{-1}a_1a_2\phi(y)=\phi(x)^{-1}\phi(y)\cdot \phi(y)^{-1}a_1\phi(y)\cdot \phi(y)^{-1}a_2\phi(y)\sim_{(D^2D^{-1})^{-2}}\phi(x)^{-1}\phi(y).$$
\end{proof}

Recall that the commutator set of a group $G$ is defined as: $\C(G,G)=\{[x,y]:x,y\in G\}$.
\begin{lemma}\label{Lem: ComIsPre}
    Let $\phi:G\to H$ be a quasi-homomorphism with $D=D(\phi)$. Then there exists a constant $C$ such that $\phi(\C(G,G))\subset \mathcal N_C(\C(H,H))$.
\end{lemma}
\begin{proof}
    For any $x,y\in G$, Lemma \ref{Lem: inverse} and Lemma \ref{Lem: QHProperties} (\ref{def}), (\ref{inverse}) show that 
    \begin{align*}
    \phi(xyx^{-1}y^{-1})& \sim_{D^3}\phi(x)\phi(y)\phi(x^{-1})\phi(y^{-1})\\
    &\sim_{(D^2D^{-1})^{-2}}\phi(x)\phi(y)\phi(x)^{-1}\phi(y^{-1})\\
    &\sim_{D^{-2}}\phi(x)\phi(y)\phi(x)^{-1}\phi(y)^{-1}\in \C(H,H).
    \end{align*}
    Then the conclusion follows since $D$ is a finite set in $H$.
\end{proof}

%\begin{lemma}\label{Lem: SujQHImpEle}
%    Let $\phi: G\to H$ be a surjective quasi-homomorphism. Then $\Delta_{\phi}$ is a virtually abelian normal subgroup of $H$.
%\end{lemma}
%\begin{proof}
%    Since $\phi$ is surjective, Lemma \ref{Lem: QHProperties} (\ref{normalizer}) shows that  $H=\phi(G)=N_H(\Delta_{\phi})$. Thus $\Delta_{\phi}\lhd H$. By Lemma \ref{Lem: QHProperties} (\ref{centralizer}), there exists a constant $C$ such that $\phi(G)=H\subset \mathcal N_C(Z_H(\Delta_{\phi}))$. This implies that $[H: Z_H(\Delta_{\phi})]<\infty$. Denote $\Delta_{\phi}'=\Delta_{\phi}\cap Z_H(\Delta_{\phi})$. Note that $\Delta_{\phi}'$ is an abelian group. According to the isomorphism theorem of groups, $[\Delta_{\phi}: \Delta_{\phi}']\le [H: Z_H(\Delta_{\phi})]<\infty$, which implies that $\Delta_{\phi}$ is virtually abelian.
%\end{proof}

For a map $\phi: G\to H$, one can define another kind of defect set of $\phi$ as $$\tilde D(\phi)=\{\phi(x)\phi(y)\phi(xy)^{-1}: x,y\in G\}.$$

\begin{lemma}\label{Lem: DualDef}\cite[Proposition 2.3]{Heu20}
    A map $\phi: G\to H$ is a quasi-homomorphism if and only if $\tilde D(\phi)$ is finite.
\end{lemma}
%\begin{proof}
    %By definition of quasi-homomorphisms, it is equivalent to show that $D(\phi)$ is finite if and only if $\tilde D(\phi)$ is finite. By symmetry, we only need to show that $\tilde D(\phi)$ is finite when $D(\phi)$ is finite.

    %Denote $D=D(\phi)$. Lemma \ref{Lem: QHProperties}  (\ref{conjugate}) shows that $h^{-1}Dh\subset D^2D^{-1}$ for any  $h\in \phi(G)$. Thus, $h^{-1}D^{-1}h\subset DD^{-2}$. Moreover, any $h\in \phi(G)$ can be written as $h=\phi(x^{-1})$ for some $x^{-1}\in G$. By Lemma \ref{Lem: QHProperties} (\ref{inverse}), $\phi(x^{-1})\sim_{D^{-2}}\phi(x)^{-1}$. Together with the above discussions, one has that $$hD^{-1}h^{-1}=\phi(x^{-1})D^{-1}\phi(x^{-1})^{-1}\subset \phi(x)^{-1}D^{-3}D^2\phi(x)\subset D'$$ where $D'=(DD^{-2})^3(D^2D^{-1})^2$.  
    
    %For any $x,y\in G$, there exists $s\in D$ such that $\phi(xy)=\phi(x)\phi(y)s$. Therefore, $$\phi(x)\phi(y)\phi(xy)^{-1}=\phi(x)\phi(y)s^{-1}\phi(y)^{-1}\phi(x)^{-1}\in \phi(x)D'\phi(x)^{-1}\subset D''$$ where $D''=(D'^{-1}D'^{2})^3(D'^{-2}D')^2$ is a finite set.Since $x,y\in G$ are arbitrary, the above formula shows that $\tilde D(\phi)\subset D''$, which completes the proof.
%\end{proof}

In general, a bounded perturbation of a quasi-homomorphism is not necessarily a quasi-homomorphism. However, if the perturbation is achieved through finitely many central elements, then it remains a quasi-homomorphism.
\begin{lemma}\label{Lem: BddPer}
    Let $\phi: G\to H$ be a quasi-homomorphism and $\phi': G\to H$ be a map. If there exists a finite subset $A\subset Z(H)$ such that $\phi'(g)\sim_A\phi(g)$ for any $g\in G$, then $\phi'$ is a quasi-homomorphism equivalent to $\phi$.
\end{lemma}
\begin{proof}
    Let $D=D(\phi)$ be the defect set of $\phi$. It is easy to see that $D(\phi')\subset A^{-2}DA$ and thus $\phi'$ is also a quasi-homomorphism.
\end{proof}

Let $(G,H)$ be a group pair and $T$ be a right coset transversal of $H$ in $G$. It is clear that any element $g\in G$ can be written uniquely as $g=g_Hg_T$ where $g_H\in H$ and $g_T\in T$. The following simple lemma  will be used in later proofs.
\begin{lemma}\label{Lem: ProdDecomp}
    Let $(G,H)$ be a group pair and $T$ be a right coset transversal of $H$ in $G$. If $H$ is a normal subgroup of $G$, then the product of any two elements $g,g'\in G$ satisfies that $(gg')_H=g_H(g_Tg'_Hg_T^{-1})(g_Tg'_T)_H$ and $(gg')_T=(g_Tg'_T)_T$. Moreover, $(g_Tg'_T)_H\in \tilde D(s)$ where $s: G/H\to G$ is a section such that $T=s(G/H)$.
\end{lemma}
\begin{proof}
    It is straightforward to calculate that $$gg'=g_Hg_Tg'_Hg'_T=g_H(g_Tg'_Hg_T^{-1})g_Tg'_T=g_H(g_Tg'_Hg_T^{-1})(g_Tg'_T)_H(g_Tg'_T)_T.$$
    Since $H\lhd G$, $g_Tg'_Hg_T^{-1}\in H$. By the uniqueness of decomposition, one gets that $(gg')_H=g_H(g_Tg'_Hg_T^{-1})(g_Tg'_T)_H$ and $(gg')_T=(g_Tg'_T)_T$. Moreover, let $\pi: G\to G/H$ be the quotient map and $s: G/H\to G$ be a section such that $T=s(G/H)$. This means that for any $g=g_Hg_T\in G$, $s(\pi(g))=s(\pi(g_T))=g_T$. Then $s(\pi(gg'))=(gg')_T=(g_Tg'_T)_T$. Finally, it follows from $g_Tg'_T=(g_Tg'_T)_H(g_Tg'_T)_T$ that $$(g_Tg'_T)_H=g_Tg'_T(g_Tg'_T)_T^{-1}=s(\pi(g))s(\pi(g'))s(\pi(gg'))^{-1}\in \tilde D(s).$$
\end{proof}

\subsection{Elementary properties of quasi-isomorphisms}

Recall that a quasi-homomorphism $\phi: G\to H$ is a quasi-isomorphism if there exists a quasi-homomorphism $\phi': H\to G$, called the quasi-inverse of $\phi$ such that $\d(\phi'\circ \phi, Id_G)<\infty$ and $\d(\phi\circ \phi', Id_H)<\infty$. Moreover, a quasi-isomorphism is called strict if $\phi'=\phi^{-1}$.

\begin{lemma}\label{Lem: EquiRelation}
    Quasi-isomorphism and strict quasi-isomorphism are two equivalent relations on discrete groups.
\end{lemma}
\begin{proof}
    Both reflexivity and symmetry are easy to see. We only need to verify the transitivity.
    Let $\phi_1: G\to H$ and $\phi_2: H\to K$ be two quasi-isomorphisms between discrete groups. By definition, there exist two quasi-homomorphisms $\phi_1': H\to G$ and $\phi_2': K\to H$ and four finite subsets $A_1\subset G, B_1,A_2\subset H, B_2\subset K$ such that for any $g\in G, h\in H, k\in K$, $$\phi_1'(\phi_1(g))\sim_{A_1}g, \quad \phi_1(\phi_1'(h))\sim_{B_1}h, \quad \phi_2'(\phi_2(h))\sim_{A_2}h, \quad \phi_2(\phi_2'(k))\sim_{B_2}k.$$ 
    It follows that $$\phi_2'\circ \phi_2(\phi_1(g))\sim_{A_2}\phi_1(g), \quad \phi_1\circ \phi_1'(\phi_2'(k))\sim_{B_1}\phi_2'(k).$$ By composing $\phi_1'$ or $\phi_2$ on the left, we get 
    \begin{equation}\label{Equ: comp1}
        \phi_1'\circ\phi_2'\circ \phi_2(\phi_1(g))\sim_{\phi_1'(A_2)D(\phi_1')}\phi_1'(\phi_1(g))\sim_{A_1}g
    \end{equation}
    and 
    \begin{equation}\label{Equ: comp2}
        \phi_2\circ\phi_1\circ \phi_1'(\phi_2'(k))\sim_{\phi_2(B_1)D(\phi_2)}\phi_2(\phi_2'(k))\sim_{B_2}k.
    \end{equation}
    Denote $\phi=\phi_2\circ \phi_1: G\to K$ and $\phi'=\phi_1'\circ \phi_2': K\to G$ as two composite maps. They are both quasi-homomorphisms since a composition of quasi-homomorphisms is still a quasi-homomorphism. The above formulas (\ref{Equ: comp1}) and (\ref{Equ: comp2}) show that both $\{g^{-1}\phi'(\phi(g)): g\in G\}$ and $\{k^{-1}\phi(\phi'(k)): k\in K\}$ are finite. Therefore, $\phi, \phi'$ are quasi-isomorphisms, which shows that quasi-isomorphism is an equivalent relation on discrete groups. For the other relation, the same proof works.
\end{proof}

 A group $G$ is called \textit{almost commutative} if $\C(G,G)$ is finite. As a corollary of Lemma \ref{Lem: ComIsPre} and Lemma \ref{Lem: EquiRelation}, we have
\begin{lemma}\label{Lem: AC}
    Almost commutativity is a quasi-isomorphism invariant for discrete groups.
\end{lemma}

In the field of Geometric Group Theory, people always care about quasi-isometry invariants for finitely generated groups. The following lemma implies that quasi-isometry invariants are also quasi-isomorphism invariants.

\begin{lemma}\label{Lem: QIToQI}
    For finitely generated groups equipped with word metrics, quasi-isomorphisms are quasi-isometries.
\end{lemma}
\begin{proof}
    Let $G,G'$ be two finitely generated groups with finite generating sets $S,S'$, respectively. Let $\phi: G\to G'$ be a quasi-isomorphism and $\phi': G'\to G$ be its quasi-inverse. By definition, there exists a constant $C>0$ such that $\sup_{g\in G}d_S(\phi'(\phi(g)),g)\le C$ and $\sup_{g'\in G'}d_{S'}(\phi(\phi'(g')),g')\le C$. The second inequality shows that $G'\subseteq \mathcal N_C(\phi(G))$. Thus it remains to show that $\phi: G\to G'$ is a quasi-isometric embedding.

    Let $D=D(\phi)$ and $D'=D(\phi')$ be the corresponding defect sets of $\phi$ and $\phi'$, respectively. Denote $C_1=\max_{x\in D'}d_S(1,x)+\max_{y\in D}d_{S'}(1,y)$ and $C_2=\max_{s'\in S'}d_S(1,\phi'(s'))+\max_{s\in S}d_{S'}(1,\phi(s))$. For any $g\in G$ with $d_S(1,g)=n$, it follows from Lemma \ref{Lem: QHProperties} (\ref{def}) that $\phi(g)\in \phi(S)^nD^{n-1}$. Then by triangle inequality, one gets that $d_{S'}(1,\phi(g))\le (C_1+C_2)d_S(1,g)-C_1$. In general, for any $g_1,g_2\in G$, Lemma \ref{Lem: inverse} shows that there exists $a\in D^{-1}(D^2D^{-1})^2$ such that $\phi(g_1)^{-1}\phi(g_2)=\phi(g_1^{-1}g_2)a$. Therefore, 
    \begin{align*}
        d_{S'}(\phi(g_1),\phi(g_2))&=d_{S'}(1,\phi(g_1)^{-1}\phi(g_2))=d_{S'}(1,\phi(g_1^{-1}g_2)a)\\ &\le d_{S'}(1,\phi(g_1^{-1}g_2))+7C_1\le (C_1+C_2)d_S(1,g_1^{-1}g_2)+6C_1=(C_1+C_2)d_S(g_1,g_2)+6C_1.
    \end{align*}
    As for the other direction, one has
    $$d_S(g_1,g_2)\le d_S(\phi'(\phi(g_1)),\phi'(\phi(g_2)))+2C\le (C_1+C_2)d_{S'}(\phi(g_1),\phi(g_2))+6C_1+2C.$$
    In summary, $\phi$ is a $(C_1+C_2, 6C_1+2C)$-quasi-isometry.
\end{proof}

The following lemma shows that those quasi-isomorphisms close to identity maps can be characterized by two special kinds of subsets in groups.
\begin{lemma}\label{Lem: GoodQH}
    Let $\phi: G\to G$ be a map from a discrete group $G$ to itself. The following are equivalent:
    \begin{enumerate}
        \item\label{2.10.1} $\phi$ is a quasi-isomorphism with a quasi-inverse given by the identity map.
        \item\label{2.10.2} The set $A=\{\phi(g)^{-1}g: g\in G\}$ is finite and $\phi(G)^{-1}$ almost commutes with $A$.
    \end{enumerate}
\end{lemma}
\begin{proof}
    $(\ref{2.10.1})\Rightarrow(\ref{2.10.2})$: By assumption, $\d(\phi, Id_G)<\infty$. This gives directly the finiteness of $A=\{\phi(g)^{-1}g: g\in G\}$. Note that for any $g\in G$, $g\sim_A\phi(g)$. Let $D=D(\phi)$ be the defect set of $\phi$. Hence, for any $g\in G, h\in \phi(G)$, $$gh\sim_A\phi(gh)\sim_D\phi(g)\phi(h)\sim_{A^{-1}}\phi(g)h.$$ By multiplying $h^{-1}\phi(g)^{-1}$ to the left, one gets that $\{h^{-1}\phi(g)^{-1}gh: g\in G, h\in \phi(G)\}\subseteq A^{-1}DA$. This implies that $\C(\phi(G)^{-1},A)$ is contained in $A^{-1}DAA^{-1}$ and thus is finite.

    $(\ref{2.10.2})\Rightarrow(\ref{2.10.1})$: The finiteness of $A$ implies $\d(\phi, Id_G)<\infty$. It remains to show that $\phi$ is a quasi-homomorphism. 
    
    Note that for any $g\in G$, $g=\phi(g)a$ for some $a\in A$. Then for any $g=\phi(g)a, g'=\phi(g')a'$, one has that
    $$\phi(gg')\sim_{A^{-1}}gg'=\phi(g)a\phi(g')a'=\phi(g)\phi(g')\cdot [\phi(g')^{-1},a]aa'\sim_{\C(\phi(G)^{-1},A)A^2}\phi(g)\phi(g').$$
    This shows that $\phi$ is a quasi-homomorphism.
\end{proof}

\section{Quasi-retractions}\label{sec: quasi-retraction}

\subsection{Elementary properties of quasi-retracts}\label{subsec: quasi-retraction}

\begin{defn}
    Let $(G,H)$ be a group pair. Classically a group homomorphism $r : G \to H$ is called a \textit{retraction} if it is left-inverse to the inclusion $\iota : H \to G$. We thus call a quasi-homomorphism $r : G \to H$ a \textit{quasi-retraction} if $\d(r\circ \iota,Id_H)<\infty$. If such a map exists, then $H$ is called a $quasi-retract$ of G. 
\end{defn}

At first, we give an equivalent definition of quasi-retraction.
\begin{lemma}\label{Lem: EquiDef}
    Let $(G,H)$ be a group pair and $\iota: H\to G$ be the inclusion map. Then $H$ is a quasi-retract of $G$ if and only if there exists a quasi-homomorphism $r: G\to H$ such that $r\circ \iota=Id_H$.
\end{lemma}
\begin{proof}
    ``$\Leftarrow$'' is obvious. We now prove ``$\Rightarrow$''. Let $\phi: G\to H$ be a quasi-retraction with the defect set $D=D(\phi)$. By definition, the set $A:=\{\phi(h)^{-1}h: h\in H\}$ is finite. Note that $h\sim_A\phi(h)$ for any $h\in H$.

    \begin{claim}
        The set $B:=\{\phi(g)^{-1}a\phi(g): g\in G, a\in A\}$ is finite in $H$.
    \end{claim}
    \begin{proof}[Proof of Claim]
        For any $g\in G, h\in H$, one has $h\phi(g)\in H$. Thus, $$h\phi(g)\sim_A\phi(h\phi(g))\sim_D\phi(h)\phi(\phi(g))\sim_{A^{-1}}\phi(h)\phi(g).$$
        By multiplying $\phi(g)^{-1}\phi(h)^{-1}$ to the left, one gets that $B$ is contained in $A^{-1}DA$ and thus is finite.
    \end{proof}
    
    Define a map $r: G\to H$ as follows:
    $$r(g)=\left\{
  \begin{array}{ll}
    g, & \hbox{$g\in H$;} \\
    \phi(g), & \hbox{$g\notin H$.}
  \end{array}
\right.$$
    By definition, $r|_H=Id_H$. It remains to show that $r$ is a quasi-homomorphism, i.e. there exists a finite subset $D'\subset H$ such that $r(gg')\sim_{D'}r(g)r(g')$ for any $g,g'\in G$.  

    \textbf{Case I: $g,g'\in H\Rightarrow gg'\in H$.} In this case, $r(gg')=gg'=r(g)r(g')$.

    \textbf{Case II: $g\in H, g'\notin H\Rightarrow gg'\notin H$.} In this case, $$r(g)r(g')=g\phi(g')=\phi(g)\phi(g')\cdot \phi(g')^{-1}\phi(g)^{-1}g\phi(g')\sim_B\phi(g)\phi(g')\sim_{D^{-1}}\phi(gg')=r(gg').$$ 

    \textbf{Case III: $g\notin H, g'\in H\Rightarrow gg'\notin H$.} In this case, $$r(g)r(g')=\phi(g)g'\sim_A\phi(g)\phi(g')\sim_{D^{-1}}\phi(gg')=r(gg').$$

    \textbf{Case IV: $g,g'\notin H$.} If $gg'\in H$, then $$r(gg')=gg'=\phi(g)\phi(g')\cdot \phi(g')^{-1}\phi(g)^{-1}g\phi(g')\cdot \phi(g')^{-1}g'\sim_{BA}\phi(g)\phi(g')=r(g)r(g').$$
    If $gg'\notin H$, then $$r(gg')=\phi(gg')\sim_D\phi(g)\phi(g')=r(g)r(g').$$

    In summary, we conclude that $r$ is a quasi-homomorphism and the conclusion follows.
\end{proof}

With the help of Lemma \ref{Lem: EquiDef}, we will frequently use the equivalent definition of quasi-retracts in what follows.

Recall that a quasi-homomorphism with values in  $\R$ equipped with the standard metric is
called a quasimorphism. A quasimorphism $\phi: G\to \R$ is called \textit{homogeneous} if $\phi(g^n)=n\phi(g)$ for all $g\in G, n\in \Z$. According to \cite[Lemma 2.21]{Cal09}, any quasimorphism $\phi: G\to \R$ has a \textit{homogenization} which is a homogeneous quasimorphism $\bar \phi: G\to \R$ defined by $\bar \phi(g)=\lim_{n\to\infty}\phi(g^n)/n$.

\begin{lemma}\label{Lem: UnbddQM}
    A group $G$ admits an unbounded quasimorphism if and only if it contains an infinite cyclic quasi-retract.
\end{lemma}
\begin{proof}
    ``$\Leftarrow$'' is clear. Now we prove ``$\Rightarrow$''. Let $\phi: G\to \R$ be an unbounded quasimorphism on $G$. Denote by $\bar \phi$ the homogenization of $\phi$. Since $\d(\bar\phi,\phi)<\infty$ (cf. \cite[Lemma 2.21]{Cal09}), $\bar\phi \nequiv 0$. Up to rescaling $\bar \phi$, we can pick an element $h\in G$ such that $\bar \phi(h)=1$. Let $\lfloor \cdot \rfloor: \R\to \Z$ be the floor function which maps a real number $x$ to the maximal integer $\le x$. Then the map $r: G\to \langle h\rangle$ defined by mapping each $g\in G$ to $r(g):=h^{\lfloor\bar\phi(g)\rfloor}$ gives a quasi-retraction.
\end{proof}

The next lemma gives some basic but very useful properties of quasi-retractions.

\iffalse
\begin{lemma}\label{Lem: QRProperties1}
    Let $r: G\to H$ be a quasi-retraction with $D=D(r)$. Denote $\tilde \Delta_r$ as the normal closure of $\Delta_r$ in $G$, i.e., $\tilde \Delta_r=\langle \langle \Delta_r\rangle\rangle$. Then
    \begin{enumerate}
        \item\label{it1} for any $g\in G$, $r(g)^{-1}g\in r^{-1}(D)$.
        \item\label{it2} $r(\tilde \Delta_r)\subset \Delta_r$.
    \end{enumerate}
\end{lemma}
\begin{proof}
    (\ref{it1}) For any $g\in G$, $r(g)^{-1}\in H$. Then $r(r(g)^{-1}g)\sim_Dr(r(g)^{-1})r(g)=r(g)^{-1}r(g)=1$.

    (\ref{it2}) Since $r$ is a surjective map, Lemma \ref{Lem: SujQHImpEle} shows that $\Delta_r$ is a normal subgroup of $H$. For any $g\in G$, Lemma \ref{Lem: QHProperties} (\ref{inverse}) shows that $r(g^{-1})\sim_{D^{-2}}r(g)^{-1}$. Then for any $s\in \Delta_r$, one has that $$r(gsg^{-1})\sim_{D^2}r(g)r(s)r(g^{-1})\sim_{D^{-2}}r(g)sr(g)^{-1}\in \Delta_r.$$ Denote $S:=\{gsg^{-1}: g\in G, s\in \Delta_r\}$. The above formula shows that $r(S)\subset \Delta_r$. Since $\tilde \Delta_r$ is generated by $S$, any element $g\in \tilde \Delta_r$ can be written as $g=s_1\cdots s_n$ where $s_i\in S, 1\le i\le n$. Therefore, $$r(g)=r(s_1\cdots s_n)\sim_{D^{n-1}}r(s_1)\cdots r(s_n)\in r(S)^n\subset \Delta_r.$$
\end{proof}
\fi

\begin{lemma}\label{Lem: QRProperties}
    Let $r: G\to H$ be a quasi-retraction with $D=D(r)$. Then the set $A:=\{r(g)^{-1}g: g\in G\}$ satisfies $r(A)\subset D$. Moreover, suppose that $G$ is generated by a finite set $S$. Then
    \begin{enumerate}
        \item\label{it3} $H$ is generated by a finite set $r(S)\cup D$.
        \item\label{it4} $r$ is a Lipschitz map with respect to $d_S$ on $G$ and any proper word metric on $H$.
        \item\label{it5} $H$ is an undistorted subgroup of $G$.
        \item\label{it6} if $G$ is of type $F_n$ (or $FP_n$), then so is $H$.
        \item\label{it7} if $[G:H]<\infty$, then $r$ is a quasi-isomorphism with a quasi-inverse given by the inclusion map $H\to G$. 
    \end{enumerate}
\end{lemma}
\begin{proof}
    For any $g\in G$, $r(g)^{-1}\in H$. Then $r(r(g)^{-1}g)\sim_Dr(r(g)^{-1})r(g)=r(g)^{-1}r(g)=1$.

    (\ref{it3}) Let $S$ be a finite generating set of $G$. For any $h\in H\le G$, we denote $h=s_1\cdots s_n$ where  $s_i\in S, 1\le i\le n$. Since $r|_H=Id_H$, $$h=r(h)=r(s_1\cdots s_n)\sim_{D^{n-1}}r(s_1)\cdots r(s_n)\subset r(S)^{n}.$$ Since $h$ is arbitrary, the above equality implies that $r(S)\cup D\subset H$ generates $H$.

    (\ref{it4}) By Item (\ref{it3}), $S'=r(S)\cup D$ is a finite generating set of $H$. Fix arbitrary two elements $g,g'\in G$. Denote $n=d_S(g,g')$. Namely, $g^{-1}g'=s_1\cdots s_n$ where $s_i\in S, 1\le i\le n$. Then $$r(g^{-1}g')=r(s_1\cdots s_n)\sim_{D^{n-1}}r(s_1)\cdots r(s_n)$$ which implies that $d_{S'}(1, r(g^{-1}g'))\le 2n-1$. Moreover, there exists $a\in D$ such that $$d_{S'}(r(g), r(g'))=d_{S'}(r(g), r(g)r(g^{-1}g')a)=d_{S'}(1, r(g^{-1}g')a)\le 2n-1+1=2d_S(g,g').$$ Since $g,g'\in G$ are arbitrary, we obtain that $r$ is a Lipschitz map with respect to $d_S$ on $G$ and $d_{S'}$ on $H$. The conclusion then follows since any two proper word metrics on $H$ are bi-Lipschitz to each other.

    (\ref{it5}) Fix a finite generating set $S$ of $G$. By Item (\ref{it3}), $S'=r(S)\cup D$ is a finite generating set of $H$. Let $S''=S\cup S'$. Then $S''$ is also a finite generating set of $G$. Observe that $r(S'')=S'$. For the inclusion map $(H,d_{S'})\to (G,d_{S''})$, the same proof of Item (\ref{it4}) shows that $d_{S'}(h,h')\le 2d_{S''}(h,h')$ for any $h,h'\in H$. Since $d_{S''}(h,h')\le d_{S'}(h,h')$ is obvious, one gets that the inclusion map is  a quasi-isometric embedding. Hence, $H$ is undistorted in $G$.

    (\ref{it6}) This follows from the combination of Item (\ref{it4}) and \cite[Theorem 8]{Alo94}.

    (\ref{it7}) Let $\iota: H\to G$ be the inclusion map. By definition of quasi-retracts, we have $r\circ \iota=Id_H$. It remains to show that $\d(\iota\circ r, Id_G)=\sup_{g\in G}d_S(r(g),g)<\infty$.

    Fix a finite right coset transversal $T$ of $H$ in $G$. Then any $g\in G$ can be written uniquely as $g=g_Hg_T$ where $g_H\in H$ and $g_T\in T$. Thus, $r(g)=r(g_Hg_T)=r(g_H)r(g_T)a=g_Hr(g_T)a$ for some $a\in D$. By setting $M:=\max_{x\in T^{-1}r(T)D}d_S(1,x)$, one obtains that $$d_S(r(g),g)=d_S(g_Hr(g_T)a,g_Hg_T)=d_S(1,g_T^{-1}r(g_T)a)\le M.$$ Since $g$ is arbitrary, we complete the proof.
\end{proof}

Recall that a finitely generated group $G$ has property (QT) if $G$ acts isometrically on a finite product of quasi-trees equipped with $\ell^1$-metric such that the orbit map is a quasi-isometric embedding. This property generalizes the notion of finite asymptotic dimensions.

\begin{corollary}\label{Cor: HypGp}
    \begin{enumerate}
        \item Quasi-retracts of hyperbolic groups are still hyperbolic groups.
        \item Property (QT) is inherited by quasi-retracts.
    \end{enumerate}
\end{corollary}
\begin{proof}
    Both items follow directly from Lemma \ref{Lem: QRProperties} (\ref{it3}) and (\ref{it5}).
\end{proof}

Since the coboundedness of group actions on hyperbolic spaces do not pass to undistorted subgroups, we raise the following question.
\begin{question}
    Is property (PH) or (PH') inherited by quasi-retracts?
\end{question}
By the way, the following Corollary \ref{Cor: NorQRHasPH'} shows that property (PH') is inherited by normal quasi-retracts.

\subsection{Quasi-retracts of hyperbolic groups}
In this subsection, we are going to give a characterization of quasi-retracts of torsion-free hyperbolic groups and prove Theorem \ref{IntroThm2} (\ref{ite2}).  For an algebraic characterization of retracts of $\F_n$, we refer the readers to \cite[Proposition 1]{Tur96}.

\begin{lemma}\label{Lem: QROfTorFreHypGp}
    Let $G$ be a torsion-free hyperbolic group and $H$ be a quasi-retract of $G$. Then $H$ is either a cyclic subgroup or a retract of $G$. 
\end{lemma}
\begin{proof}
    Let $r: G\to H$ be a quasi-retraction and $D(r)$ be the defect set. By Corollary \ref{Cor: HypGp}, $H$ is also a hyperbolic group. By Lemma \ref{Lem: SujQHImpEle}, the defect subgroup $\Delta_r=\langle D(r)\rangle$ is a virtually abelian normal subgroup of $H$. Since a hyperbolic group does not contain $\Z^2$, the defect subgroup $\Delta_r$ must be cyclic. If $\Delta_r=\{1\}$, then $r$ is a retraction. Otherwise $\Delta_r$ is generated by a hyperbolic element. Then it follows from $\Delta_r\lhd H$ that $H$ is also cyclic.
\end{proof}

\begin{remark}\label{Rmk: QROfFreeGp}
    If $G$ is a free group in Lemma \ref{Lem: QROfTorFreHypGp}, then every cyclic subgroup is also a quasi-retract of $G$. In fact, suppose that $H=\langle w\rangle$ for some $w\in G$. Up to conjugation, we can assume that $w$ is a cyclically reduced word in $G$. Let $h_w$ be the corresponding Brooks counting  quasimorphism, which is a quasimorphism on $G$ satisfying $h_w(w^n)=n$ for $n\in \Z$; see \cite[Subsection 2.3.3]{Cal09} for a reference. Then the map $r: G\to H$ defined by $r(g)=w^{h_w(g)}$ is a quasi-retraction.
\end{remark}

Recall that a quasi-retract is called \textit{nontrivial} if it is neither a finite subgroup nor the whole group. A group element $g\in G$ is called \textit{primitive} if $g$ can not be represented as a nonzero power of any $h\in G\setminus\{g, g^{-1}\}$. A group $G$ is called \textit{Hopfian} if any surjective homomorphism from $G$ to $G$ is an isomorphism. 
\begin{lemma}\label{Lem: RetractOfFreeGps}
    Let $G=\F_2$. Any nontrivial retract of $G$ is a cyclic subgroup generated by a primitive element.
\end{lemma}
\begin{proof}
    Let $H\le G$ be a nontrivial retract of $G$. That is, there exists a homomorphism $r: G\to H$ such that $r|_H=Id$. Since $r$ is a surjective homomorphism, $\rk(H)\le \rk(G)=2$. Since $H$ is nontrivial, $\rk(H)\ge 1$.

    If $\rk(H)=2$, then $H\cong \F_2$ since any subgroup of a free group is also free. Another well-known fact is that any free group of finite rank is Hopfian \cite[Theorem 3.3]{Sel99}. This implies that $r$ is an isomorphism, which contradicts with that $H\neq G$. Therefore $\rk(H)=1$.

    Let $H=\langle g\rangle$. Suppose that $g$ is not primitive, i.e. $g=h^n$ for some $h\neq 1\in G, n\ge 2$. Denote $r(h)=h^k$ for some $k\in \Z$. Then $h^n=g=r(g)=r(h^n)=r(h)^n=h^{kn}$ which implies that $k=1$. However, $r(h)=h\notin H$ which is impossible. Therefore, $g$ is a primitive element.
\end{proof}

A subgroup $H\le G$ is called \textit{malnormal} if $H\cap gHg^{-1}=\{1\}$ for any $g\in G\setminus H$.
\begin{proposition}\label{Prop: HESNotQR}
    A subgroup of $\F_2$ is a nontrivial quasi-retract if and only if it is a cyclic group. Moreover, there exists a hyperbolically embedded subgroup of $\F_2$ which is not a quasi-retract.
\end{proposition}
\begin{proof}
    If $H$ is a nontrivial quasi-retract of $\F_2$, then Lemma \ref{Lem: QROfTorFreHypGp}, together with Lemma \ref{Lem: RetractOfFreeGps} show that $H$ is a cyclic group. The opposite direction is given by Remark \ref{Rmk: QROfFreeGp}. Now we prove the ``Moreover'' part.
    In \cite[Example 1]{BMR99}, Baumslag-Miasnikov-Remeslennikov constructed an example of a finitely generated malnormal subgroup $H$ of $\F_2$ which has rank greater than 2. Hence, it can not be a quasi-retract of $\F_2$. Since any finitely generated subgroup of a free group is quasi-convex, it follows from \cite{Bow12} that $\F_2$ is hyperbolic relative to $H$. As a result of \cite[Proposition 2.4]{DGO17}, $H$ is a hyperbolically embedded subgroup of $\F_2$. 
\end{proof}

%By combining Proposition \ref{Prop: HESNotQR} with Lemma \ref{Lem: QROfFreeGps}, we complete the proof of Theorem \ref{IntroThm2} (\ref{ite2}).

\subsection{First examples}\label{subsec: examples}
The goal of this subsection is to give some general examples and non-examples of quasi-retracts. The first part aims to explore when a finite-index subgroup becomes a quasi-retract. Note that a quasimorphism on a finite-index subgroup can not be extended to the whole group in general. 

\begin{proposition}\label{Prop: FinInd}
    Let $(G,H)$ be a group pair. Suppose $H$ is of finite index in $G$. Then $H$ is a quasi-retract of $G$ if and only if there exists a right coset transversal $T$ containing 1 of $H$ in $G$ such that $H$ almost commute with $T$.
\end{proposition}
\begin{proof}
    ``$\Rightarrow$'': let $r: G\to H$ be a quasi-retraction with $D=D(r)$. Since $[G: H]<\infty$, Lemma \ref{Lem: QRProperties} (\ref{it7}) shows that $r$ is a quasi-isomorphism with a quasi-inverse given by the inclusion map $H\to G$. As a result of Lemma \ref{Lem: GoodQH}, the set $A:=\{r(g)^{-1}g: g\in G\}$ is finite and $H$ almost commutes with $A$. Note that any $g\in G$ can be written as $g=r(g)\cdot r(g)^{-1}g$ and $1\in A$. Hence, $A$ contains a right coset transversal $T$ containing $1$ of $H$ in $G$. And the finiteness of $\C(H,T)$ follows from the finiteness of $\C(H,A)$.

    ``$\Leftarrow$'': suppose that there exists a right coset transversal $T$ containing 1 of $H$ in $G$ such that $H$ almost commute with $T$. For any $g\in G$, $g$ can be uniquely written as $g=g_Hg_T$ where $g_H\in H$ and $g_T\in T$. Define a map $r: G\to H, g=g_Hg_T\mapsto g_H$. Since $1\in T$, $r|_H=Id_H$. Since $T$ is finite, the set $A:=\{r(g)^{-1}g: g\in G\}\subseteq T$ is also finite. By Lemma \ref{Lem: GoodQH}, $r$ is a quasi-homomorphism (even a quasi-isomorphism) and the conclusion then follows.
\end{proof}

We also show that a quasi-retraction to a finite-index subgroup of a subgroup is a composition of two quasi-retractions. 

\begin{lemma}\label{Lem: FinInd}
    Let $(G,H)$ be a group pair. Let $K$ be a finite-index subgroup of $H$. Then $K$ is a quasi-retract of $G$ if and only if $H$ is a quasi-retract of $G$ and $K$ is a quasi-retract of $H$.
\end{lemma}
\begin{proof}
    ``$\Leftarrow$'' is clear since a composition of two quasi-retractions is still a quasi-retraction. Now, we prove ``$\Rightarrow$''. Let $r: G\to K$ be a quasi-retraction with the defect set $D=D(r)$. Clearly $r|_H: H\to K$ is also a quasi-retraction. Since $[H:K]<\infty$,  Lemma \ref{Lem: QRProperties} (\ref{it7}) shows that $r|_H$ is a quasi-isomorphism with a quasi-inverse given by the inclusion map $K\to H$. By Lemma \ref{Lem: GoodQH}, the set $A:=\{r(h)^{-1}h: h\in H\}$ is finite. Now, we take $r: G\to K$ as a quasi-homomorphism $r: G\to H$. The finiteness of $A$ implies that $\d(r|_H, Id_H)<\infty$. By definition, $H$ is a quasi-retract of $G$.
\end{proof}

%Snopce-Tanushevski-Zalesskii proved in \cite{STZ22} that for a free group $\F_n$, there exists a finite-index subgroup $H$ and a retract $R$ of $\F_n$ such that $H\cap R$ is not a retract of $H$. This fact shows that (quasi-)retractions are not preserved by restricting to finite-index subgroups.

Before giving more general examples of quasi-retracts, we first recall some definitions. 

\begin{defn}\cite{NN50}\label{Def: AmalDirPro}
    A group $G$ is said to be the \textit{direct product} of its subgroups $H$ and $K$ with an \textit{amalgamated subgroup} $Z$, denoted by $H\times_Z K$ if the following conditions hold: (i) $G$ is generated by $H\cup K$, (ii) $H\cap K=Z$, (iii) $H$ commutes with $K$, i.e. $\C(H,K)=\{1\}$.
\end{defn}

By \cite[Lemma 2.9]{FK16}, a central extension of groups $1\to H\to G\to Q\to 1$ has a bounded Euler class if and only if the central extension is right quasi-split, i.e. there exists a section $s: Q\to G$ such that $s$ is a quasi-homomorphism.

\begin{defn}\cite[Definition 2.1]{TW25}
    A group $G$ has \textit{property (PH')} if there exist finitely many hyperbolic spaces $X_1,\ldots, X_n$ on which $G$ acts coboundedly and the induced diagonal action of $G$ on the product $\prod_{i=1}^nX_i$ with $\ell^1$-metric is proper. In this case, we also say $G$ has \textit{property (PH') from actions $G\curvearrowright X_1,\cdots,G\curvearrowright X_n$}. 
    
    A group $G$ has \textit{property (PH)} if $G$ virtually has property (PH'). 
\end{defn}

By definition, any finitely generated abelian group has property (PH).

\begin{lemma}\label{Lem: ExaOfQR}
    Let $(G,H)$ be a group pair. Then $H$ is a quasi-retract of $G$ whenever the pair satisfies one of the following:
    \begin{enumerate}
        \item\label{5.4.1} $H$ is a retract of $G$. Equivalently, $G$ has a semi-direct product structure of the form $G=K\rtimes H$. 
        \item\label{5.4.2} $G$ is an arbitrary group and $H$ is a finite subgroup. 
        \item\label{5.4.3} $G$ is an amalgamated direct product of the form $H\times_{Z}K$ and the central extension $1\to Z\to K\to K/Z\to 1$ has a bounded Euler class.
        %\item\label{5.4.4} $G$ is an amalgamated free product over a finite subgroup of the form $H\ast_{F}K$.
        \item\label{5.4.5} $G$ has property (PH) and $H\le Z(G)$. 
        \item\label{5.4.6} $G$ is a finitely generated abelian group and $H$ is an arbitrary subgroup.
        \item\label{5.4.7} $H\le Z(G)$ and the central extension $1\to H\to G\to G/H\to 1$ has a bounded Euler class.
    \end{enumerate}
\end{lemma}
\begin{proof}
    (\ref{5.4.1}) is clear since a homomorphism is obviously a quasi-homomorphism. (\ref{5.4.2}) is clear by mapping every element in $H$ to itself and every element outside $H$ to $1$. (\ref{5.4.7}) is given by \cite[Lemma 5.6]{TW25}. By \cite[Theorem 1.7]{TW25}, (\ref{5.4.5}) is transformed into (\ref{5.4.7}). (\ref{5.4.6}) follows directly from (\ref{5.4.5}). Finally, it remains to show (\ref{5.4.3}). Since the central extension $1\to Z\to K\to K/Z\to 1$ has a bounded Euler class, there exists a section $s: K/Z\to K$ such that $s$ is a quasi-homomorphism. By Lemma \ref{Lem: BddPer}, we can assume that this section is normalized, i.e. $s(1)=1$. Fix a right coset transversal $T$ containing $1$ of $Z$ in $H$. By \cite[Corollary 6]{KM82}, any element $g\in H\times_Z K$ can be written uniquely as $g=hkz$ where $h\in T, k\in s(K/Z), z\in Z$. Define a map $r: G\to H$ by sending $g=hkz\mapsto hz$. Clearly $r|_{H}=Id_H$. Since $H$ commutes with $K$, it is direct to compute that $D(r)\subseteq D(s)^{-1}$. Hence, $r$ is a quasi-homomorphism and then quasi-retraction.%It remains to show (\ref{5.4.4}). Fix a right coset transversal $T_H$ (resp. $T_K$) containing $1$ of $F$ in $H$ (resp. $K$). By the classical Bass-Serre theory, any element $g\in H\ast_FK$  can be written uniquely as $g=fh_1k_1\cdots h_rk_r$ where $f\in F, h_1\in T_H, h_2,\cdots, h_r\in T_H\setminus\{1\}, k_1,\cdots,k_{r-1}\in T_K\setminus\{1\}, k_r\in T_K$. Define a map $r: G\to H$ by sending $g=fh_1k_1\cdots h_rk_r\mapsto fh_1\cdots h_r$. Clearly $r|_{H}=Id_H$ and it is easy to see that $D(r)\subset F$. Since $F$ is a finite group, $r$ is a quasi-homomorphism and then quasi-retraction.
\end{proof}

We also collect some concrete non-examples of quasi-retractions. 

\begin{example}\label{Exa: Noquasi-retraction}
    \begin{enumerate}
        \item\label{DihedralGp} Let $G=\langle t,s\mid s^2=(st)^2=1\rangle$ be an infinite dihedral group and $H=\langle t\rangle$. Since $[G:H]=2$ and $H$ does not almost commute with any right coset transversal $T$ of $H$ in $G$, Proposition \ref{Prop: FinInd} implies that $H$ is not a quasi-retract of $G$. 
        \item\label{HeisenbergGp} Let $G=\langle a,b,c\mid [a,c]=[b,c]=1,[a,b]=c\rangle$ be the Heisenberg group and $H=\langle c\rangle=Z(G)$. Since $H$ is distorted in $G$, Lemma \ref{Lem: QRProperties} (\ref{it5}) implies that $H$ is not a quasi-retract of $G$. However, $G$ quasi-retracts to $\langle a\rangle$ since $G$ has a homomorphism to $\langle a\rangle$ by sending $a, b, c$ to $a, 1, 1$ respectively.
        \item\label{BSGp} Let $G=BS(1,n)=\langle a,t\mid tat^{-1}=a^n\rangle, n\neq 0,1$ be a Baumslag-Solitar group and $H=\langle a\rangle$. Since $H$ is distorted in $G$, Lemma \ref{Lem: QRProperties} (\ref{it5}) implies that $H$ is not a quasi-retract of $G$. However, $G$ quasi-retracts to $\langle t\rangle$ since $G$ has a homomorphism to $\langle t\rangle$ by sending $t,a$ to $t,1$ respectively.
    \end{enumerate}
\end{example}

\section{Quasi-split short exact sequences}\label{sec: QSSES}

Recall from the Introduction that,  a short exact sequence $1\to H\to G\to Q\to 1$ is left split if $H$ is a retract of $G$ and right split if it admits a section $s: Q\to G$ that is a homomorphism. Comparing with the classical notion, we say the short exact sequence is \textit{left quasi-split} if $H$ is a quasi-retract of $G$ and \textit{right quasi-split} if it admits a section $s: Q\to G$ that is a quasi-homomorphism. One simple observation is that whenever $H$ or $Q$ is finite, the above short exact sequence is always right quasi-split since the defect set of any section is a finite subset of $H$. 

Recall in the classical group theory, we have the following results.

\begin{theorem}[Classical version]\label{Lem: ClaVer2}
    Let $1\to H\xrightarrow{\iota} G\xrightarrow{\pi} Q\to 1$ be a short exact sequence of groups. The following are equivalent: 
    \begin{enumerate}
        \item The short exact sequence $1\to H\to G\to Q\to 1$ is left split. Or equivalently, $H$ is a retract of $G$.
        \item There exists a section $s: Q\to G$ such that $s$ is a homomorphism and $H$ commutes with $s(Q)$.
        \item There is an isomorphism $\phi: G\to H\times Q$ such that the diagram commutes: 
        $$\xymatrix{
  1  \ar[r]^{} & H \ar[d]_{Id} \ar[r]^{\iota} & G \ar[d]_{\phi} \ar[r]^{\pi} & Q \ar[d]_{Id} \ar[r]^{} & 1  \\
  1 \ar[r]^{} & H \ar[r]^{} & H\times Q \ar[r]^{} & Q \ar[r]^{} & 1   }$$
  where the bottom row is the short exact sequence for a direct product.
    \end{enumerate}
\end{theorem}

Before giving the analogue of Theorem \ref{Lem: ClaVer2}, we need the following two lemmas.
Recall that for any two subsets $H, T$ of a group $G$, we denote $\C(H,T):=\{[h,t]=hth^{-1}t^{-1}: h\in H, t\in T\}$.

\begin{lemma}\label{Lem: SymCommutatorSet}
    Let $G$ be a group and $H\lhd G$ be a normal subgroup. For any subset $T\subset G$, $\C(H,T)=\C(T^{-1},H)=\C(H, T^{-1})^{-1}$. 
\end{lemma}
\begin{proof}
    Let $T\subset G$ be a subset. Since $[t^{-1},h]=t^{-1}hth^{-1}=(ht^{-1}h^{-1}t)^{-1}=[h,t^{-1}]^{-1}$ for all $h\in H, t\in T$, it is clear that $\C(T^{-1},H)=\C(H,T^{-1})^{-1}$. Thus we only need to verify that $\C(H,T)=\C(T^{-1},H)$. Since $H\lhd G$, $tht^{-1}, t^{-1}ht\in H$ for any $h\in H, t\in T$. Denote $h^t:=tht^{-1}$ and $h^{t^{-1}}=t^{-1}ht$. Then $(h^{-1})^t=th^{-1}t^{-1}=(h^t)^{-1}$ and $(h^{-1})^{t^{-1}}=t^{-1}h^{-1}t=(h^{t^{-1}})^{-1}$. Therefore, $$[h,t]=hth^{-1}t^{-1}=t^{-1}(tht^{-1})t(th^{-1}t^{-1})=t^{-1}h^tt(h^{-1})^t=[t^{-1},h^t]\in \C(T^{-1},H)$$ and similarly $$[t^{-1},h]=t^{-1}hth^{-1}=(t^{-1}ht)t(t^{-1}h^{-1}t)t^{-1}=h^{t^{-1}}t(h^{-1})^{t^{-1}}t^{-1}=[h^{t^{-1}},t]\in \C(H,T).$$
    Since $h\in H,t\in T$ are arbitrary, the above formulas show that $\C(H,T)\subset \C(T^{-1},H)\subset \C(H,T)$. Thus we get $\C(H,T)=\C(T^{-1},H)$ and the conclusion follows.
\end{proof}

The following lemma gives a useful necessary condition when a subgroup is a quasi-retract. 

\begin{lemma}\label{Thm: IfQR}
    Let $(G,H)$ be a group pair. If $r: G\to H$ is a quasi-retraction with $D=D(r)$, then for any normal subgroup $H_0\lhd G $ contained in $H$, there exists a right coset transversal $T$ containing $1$ of $H$ in $G$ such that $r(T)\subset D$ and $\C(H_0,T)$ is finite. 
\end{lemma}
\begin{proof}
    Since $r(1)=1=r(gg^{-1})\sim_Dr(g)r(g^{-1})$, one has that $r(g^{-1})\sim_{D^{-1}}r(g)^{-1}$ for any $g\in G$. Let $T$ be an arbitrary right coset transversal containing $1$ of $H$ in $G$. For each  $t\in T$, we denote $ t':=r(t)^{-1}t$ and $ T'=\{t': t\in T\}$. Clearly  $T'$ gives another right coset transversal containing $1$ of $H$ in $G$. And Lemma \ref{Lem: QRProperties} shows that $r(T')\subset D$.

    Since $H_0\lhd G$ and $r|_{H_0}=Id_{H_0}$, one has that for any $h\in H_0, t'\in T'$, $$t'ht'^{-1}=r(t'ht'^{-1})\sim_{D^2}r(t')r(h)r(t'^{-1})=r(t')hr(t'^{-1})\sim_{D^{-1}}r(t')hr(t')^{-1}\sim_{D^{-1}}r(t')h.$$  By multiplying $h^{-1}r(t')^{-1}$ to the left, the above formula transforms into $$h^{-1}r(t')^{-1}t'ht'^{-1}\sim_{D^{-2}D^2}1.$$ Denote $\tilde t:=r(t')^{-1}t'$. Therefore, we get 
    \begin{equation*}
        [h^{-1}, \tilde t]=h^{-1}\tilde t h \tilde t^{-1}=h^{-1}r(t')^{-1}t'ht'^{-1}r(t')\sim_{D}h^{-1}r(t')^{-1}t'ht'^{-1}\sim_{D^{-2}D^2}1.
    \end{equation*}
    Since $h\in H_0$ is arbitrary and $\tilde T=\{\tilde t: t\in T\}$ is still a right coset transversal containing $1$ of $H$ in $G$, the conclusion follows.
\end{proof}

\begin{remark}
    We remark that the inverse of Lemma \ref{Thm: IfQR} is not true. For example, let $G$ be a Heisenberg group and $H$ be its center. Then $\C(H_0,T)=\{1\}$ since $H_0\le H=Z(G)$. However, Example \ref{Exa: Noquasi-retraction} (\ref{HeisenbergGp}) shows that $H$ is not a quasi-retract of $G$.
\end{remark}

Recall that a quasi-homomorphism $\phi: G\to G'$ is quasi-isomorphism if there exists a quasi-homomorphism $\phi': G'\to G$, called the quasi-inverse of $\phi$, such that $\d(\phi'\circ \phi, Id_{G})<\infty$ and $\d(\phi\circ \phi', Id_{G'})<\infty$. Moreover, if $\phi'=\phi^{-1}$, then $\phi$ is a strict quasi-isomorphism.

\begin{theorem}\label{Thm: LQS}
    Let $1\to H\xrightarrow{\iota} G\xrightarrow{\pi} Q\to 1$ be a short exact sequence of groups.  The following are equivalent: 
    \begin{enumerate}
        \item\label{LQS} The short exact sequence $1\to H\to G\to Q\to 1$ is left quasi-split. Or equivalently, $H$ is a quasi-retract of $G$.
        \item\label{RQS} There exists a normalized section $s: Q\to G$ such that $s$ is a quasi-homomorphism and $H$ almost commutes with $s(Q)$: $\C(H,s(Q))$ is finite.
        \item\label{SQI} There is a strict quasi-isomorphism $\phi: G\to H\times Q$ such that the diagram commutes: 
        $$\xymatrix{
  1  \ar[r]^{} & H \ar[d]_{Id} \ar[r]^{\iota} & G \ar[d]_{\phi} \ar[r]^{\pi} & Q \ar[d]_{Id} \ar[r]^{} & 1  \\
  1 \ar[r]^{} & H \ar[r]^{} & H\times Q \ar[r]^{} & Q \ar[r]^{} & 1   }$$
  where the bottom row is the short exact sequence for a direct product.
    \end{enumerate}
\end{theorem}
\begin{proof}
    ``$(\ref{RQS})\Rightarrow(\ref{LQS})$'': Let $s: Q\to G$ be a normalized section such that $s$ is a quasi-homomorphism and $\C(H,s(Q))$ is finite. Denote $T=s(Q)$. Since $s$ is a normalized section, $T$ contains $1$ and any group element $g\in G$ can be uniquely written as $g=g_Hg_T$ where $g_H\in H$ and $g_T=s(\pi(g))\in T$. Define a map $r: G\to H, g=g_Hg_T\mapsto g_H$. Since $1\in T$, $r(h)=h$ for all $h\in H$. It suffices for us to show that $r$ is a quasi-homomorphism. 
    
    For any $g=g_Hg_T,g'=g'_Hg'_T\in G$, one has that $gg'=g_Hg_Tg'_Hg'_T=(gg')_H(gg')_T$. Then Lemma \ref{Lem: ProdDecomp} shows that 
    $$r(gg')=(gg')_H=g_H(g_Tg'_Hg_T^{-1})(g_Tg'_T)_H=g_Hg'_H\cdot [{g'_H}^{-1},g_T](g_Tg'_T)_H\sim_{\C(H,T)\tilde D(s)}g_Hg'_H=r(g)r(g').$$
    The above equality shows that the defect set $D(r)$ of $r$ is contained in $\C(H,T)\tilde D(s)$ and thus is finite.

    ``$(\ref{LQS})\Rightarrow(\ref{RQS})$'':  By definition, there exists a quasi-retraction $r$ from $G$ to $H$. Denote by $D=D(r)$ the finite defect set of $r$. According to Lemma \ref{Thm: IfQR}, there exists a right coset transversal $T$ containing 1 of $H$ in $G$ such that $r(T)\subset D$ and $\C(H,T)$ is finite. Let $s: Q\to G$ be the section such that $s(Q)=T$. 
    
    It remains to verify that $s$ is a quasi-homomorphism. Note that for any $g\in G$, $r(g^{-1})\sim_{D^{-1}}r(g)^{-1}$ and for any $x,y\in Q$,  $r(s(x)), r(s(y)), r(s(xy))\in r(T)\subset D$. Since $s$ is a section, $s(y)^{-1} s(x)^{-1} s(xy)\in H$. Therefore, 
    \begin{align*}
        s(y)^{-1} s(x)^{-1} s(xy)& =r( s(y)^{-1} s(x)^{-1} s(xy))\sim_{D^2}r( s(y)^{-1})r( s(x)^{-1})r(s(xy))\\ &\sim_{D}r( s(y)^{-1})r( s(x)^{-1})\sim_{D^{-1}}r( s(y)^{-1})r( s(x))^{-1}\sim_{D^{-1}}r(s(y)^{-1})\sim_{D^{-2}}1.
    \end{align*}
    Since $x,y\in Q$ are arbitrary, the above formula shows that the defect of $s$ is finite, which means that $s$ is a quasi-homomorphism. 

    ``$(\ref{RQS})\Rightarrow(\ref{SQI})$'': By assumption, there exists a normalized section $s: Q\to G$ such that $s$ is a quasi-homomorphism and $\C(H,s(Q))$ is finite. Denote $T=s(Q)$. Any group element $g$ can be uniquely written as $g=g_Hg_T$ where $g_H\in H$ and $g_T\in T$. According to the proof of ``$(\ref{RQS})\Rightarrow(\ref{LQS})$'', the map $r: G\to H, g=g_Hg_T\mapsto g_H$ is a quasi-retraction. Define two maps $\phi: G\to H\times Q$ and $\phi': H\times Q\to G$ as follows: for any $g\in G, (h,x)\in H\times Q$, $$\phi(g)=(r(g),\pi(g)), \quad \phi'(h,x)=h\cdot s(x).$$ Since $r$ is a quasi-retraction and $\pi$ is a homomorphism, it is clear that $\phi$ is a quasi-homomorphism. As for $\phi'$, we compute directly that 
    \begin{align*}
    &\phi'(h_2,x_2)^{-1}\phi'(h_1,x_1)^{-1}\phi'(h_1h_2,x_1x_2) =s(x_2)^{-1}h_2^{-1}s(x_1)^{-1}h_2s(x_1x_2)\\ &=s(x_2)^{-1}[h_2^{-1},s(x_1)^{-1}]s(x_2)s(x_2)^{-1}s(x_1)^{-1}s(x_1x_2)\sim_{D(s)}s(x_2)^{-1}[h_2^{-1},s(x_1)^{-1}]s(x_2)\\ &=[s(x_2)^{-1}, [h_2^{-1},s(x_1)^{-1}]][h_2^{-1},s(x_1)^{-1}]\sim_{\C(H,T^{-1})}[s(x_2)^{-1}, [h_2^{-1},s(x_1)^{-1}]]\sim_{\C(T^{-1},H)}1
    \end{align*}
    for any $(h_1,x_1),(h_2,x_2)\in H\times Q$. By Lemma  \ref{Lem: SymCommutatorSet}, $\C(H,T^{-1})$ and $\C(T^{-1},H)$ are both finite since $\C(H,T)$ is finite. Therefore, $\phi'$ is also a quasi-homomorphism.

    It remains to show that $\phi'=\phi^{-1}$. For any $g\in G$, $\phi'(\phi(g))=\phi'(r(g),\pi(g))=r(g)s(\pi(g))=g_Hg_T=g$. For any $(h,x)\in H\times Q$, $\phi(\phi'(h,x))=(r(hs(x)),\pi(hs(x)))=(h,x)$. Thus $\phi'=\phi^{-1}$, which completes the proof.

    ``$(\ref{SQI})\Rightarrow(\ref{LQS})$'': Let $\phi: G\to H\times Q$ be a strict quasi-isomorphism and $p: H\times Q\to H$ be the natural projection. Denote $r=p\circ \phi: G\to H$. Clearly $r$ is a quasi-homomorphism since $r$ is a composition of two quasi-homomorphisms. For any $h\in H$, it follows from the commutative diagram that $\phi(h)=(h,1)$. Thus $r(h)=p(h,1)=h$, which shows that $r$ is a quasi-retraction.
\end{proof}

\begin{remark}
    In the classical Theorem \ref{Lem: ClaVer2}, the direct product structure of $G$ guarantees that there exists a section $s: Q\to G$ which is a homomorphism. However, different to the classical version, the finiteness of $\C(H,s(Q))$ does not imply that $s$ is a quasi-homomorphism in Theorem \ref{Thm: LQS}. For example, let $G$ be the discrete Heisenberg group and $H=Z(G)$. Then $\C(H,s(G/H))=\{1\}$ for any normalized section $s: G/H\to G$. However, Example \ref{Exa: Noquasi-retraction} (\ref{HeisenbergGp}) shows that the central extension $1\to H\to G\to G/H\to 1$ is not left quasi-split. Thus as a result of Theorem \ref{Thm: LQS}, any normalized section $s: G/H\to G$ can not be a quasi-homomorphism. By the way, Example \ref{Exa: Noquasi-retraction} (\ref{DihedralGp}) shows that one can not drop the requirement that $\C(H,s(Q))$ is finite in Theorem \ref{Thm: LQS}, either.
\end{remark}

The following result is a complement to Lemma \ref{Lem: QROfTorFreHypGp} which initially characterizes quasi-retracts of torsion-free hyperbolic groups. 
\begin{proposition}\label{Prop: NorQROfHypGp}
    \begin{enumerate}
        \item\label{3.18.1} A normal quasi-retract of a hyperbolic group is either finite or of finite index. 
        \item\label{3.18.2} A normal quasi-retract of a non-elementary torsion-free hyperbolic group is trivial. 
    \end{enumerate}
\end{proposition}
\begin{proof}
    (\ref{3.18.1}) Let $G$ be a finitely generated hyperbolic group and $H\lhd G$ be a quasi-retract. By Lemma \ref{Lem: QRProperties} (\ref{it3}), $H$ is finitely generated. By Theorem \ref{Thm: LQS}, $G$ is strictly quasi-isomorphic to $H\times G/H$. By Lemma \ref{Lem: QIToQI}, $G$ is quasi-isometric to $H\times G/H$. Since word-hyperbolicity is a quasi-isometry invariant for finitely generated groups. we know that $H\times G/H$ is also hyperbolic. A well-known fact is that a hyperbolic group can contain neither an infinite torsion subgroup nor $\Z^2$. This forces either $H$ is finite or $G/H$ is finite. Then the conclusion follows.

    (\ref{3.18.2}) Let $G$ be a finitely generated non-elementary torsion-free hyperbolic group and $H\lhd G$ be a quasi-retract. By Item (\ref{3.18.1}), it suffices to show that $G/H=\{1\}$ if $[G:H]<\infty$. By Theorem \ref{Thm: LQS}, there exists a normalized section $s: G/H\to G$ such that $H$ almost commutes with $s(G/H)$. 
    
    For each hyperbolic element $f\in G$, we denote $E(f):=\{g\in G: \exists \ n\ge 1, \text{ s.t. } gf^ng^{-1}=f^n\}$. It is well-known that in a non-elementary torsion-free hyperbolic group $G$, $E(f)$ is infinite cyclic for each hyperbolic element $f\in G$ and 
    \begin{equation*}
        \bigcap_{f: \text{ hyperbolic in }G}E(f^k)=\{1\}
    \end{equation*}
    for each $k\ge 1$.

    Since $[G:H]<\infty$, there exists $k\ge 1$ such that $g^k\in H$ for all $g\in G$. Fix an arbitrary hyperbolic element $f\in G$. Then $f^k\in H$. It follows from the finiteness of $\C(H,s(G/H))$ that $\C(\langle f^k\rangle, s(G/H))$ is finite. This implies that $s(G/H)\subset E(f^k)$. Since $f\in G$ is arbitrary, one gets that $s(G/H)\subseteq \bigcap_{f: \text{ hyperbolic in }G}E(f^k)=\{1\}$. Then the conclusion follows.
\end{proof}

Let $(G,H)$ be a group pair. Recall from \cite[Definition 2.1]{DGO17} that $H$ is \textit{hyperbolically embedded} in $G$ if there exists a subset $X\subseteq G$ such that the following conditions hold:
\begin{itemize}
    \item[(i)] $G$ is generated by $X\cup H$.
    \item[(ii)] The Cayley graph $\G(G,X\cup H)$ is hyperbolic.
    \item[(iii)] $(H,\hat d)$ is a proper metric space where $\hat d$ is a relative metric defined on $\G(G,X\cup H)$.
\end{itemize}

\begin{proposition}\label{Prop: HypEmbSubpgQR}
    A finite-by-$\Z^m$ hyperbolically embedded subgroup is a quasi-retract.
\end{proposition}
\begin{proof}
    Let $G$ be a group and $H\le G$ be a finite-by-$\Z^m$ hyperbolically embedded subgroup. By assumption, $H$ fits into the following short exact sequence of groups: 
    \begin{equation}\label{equ: finite-by-abelian}
        1\to F\to H \xrightarrow{\pi} \Z^m\to 1
    \end{equation}
    where $F\le H$ is a finite group. Since a finite subgroup is always a quasi-retract, the above short exact sequence (\ref{equ: finite-by-abelian}) is left quasi-split. By Theorem \ref{Thm: LQS}, there exists a normalized section $s: \Z^m\to H$ such that $s$ is a quasi-homomorphism and $\C(F,s(\Z^m))\subset F$ is finite. If $m=0$, then $H$ is a finite subgroup which is obviously a quasi-retract. Now we assume $m\ge 1$. Let $\bar c_1,\cdots,\bar c_m$ be a standard generating set of $\Z^m$ and denote $c_i=s(\bar c_i)$ for $1\le i\le m$. Then any $g\in H$ can be uniquely represented as $g=fc_1^{n_1}\cdots c_m^{n_m}$ for some $f\in F, n_1,\cdots,n_m\in \Z$. 
    
    Fix $i\in \{1,\cdots,m\}$. The projection map $\phi_i: H\to \Z, \quad g=fc_1^{n_1}\cdots c_m^{n_m}\mapsto n_i$ factors through $H\xrightarrow{\pi} \Z^m\to \Z$ and thus is a group homomorphism from $H$ to $\Z$. Since $H$ is hyperbolically embedded in $G$, the result of Hull-Osin (i.e. \cite[Theorem 1.4]{HO13}) shows that the homomorphism $\phi_i$ can be extended to a quasimorphism $\tilde \phi_i: G\to \R$ satisfying $\tilde \phi_i|_H=\phi_i$. Let $\lfloor \cdot \rfloor: \R\to \Z$ be the floor function which maps a real number $x$ to the maximal integer $\le x$. Clearly $\lfloor \cdot \rfloor$ is a quasimorphism and so is the composite map $\rho_i=\lfloor \cdot \rfloor\circ \tilde \phi_i$. 
    Note that $\rho_i|_H=\phi_i$. This implies that $h\sim_Fc_1^{\rho_1(h)}\cdots c_m^{\rho_m(h)}$ for any $h\in H$.
    
    Define a map $r: G\to H$ as follows:
    $$r(g)=\left\{
  \begin{array}{ll}
    g, & \hbox{$g\in H$;} \\
    c_1^{\rho_1(g)}\cdots c_m^{\rho_m(g)}, & \hbox{$g\notin H$.}
  \end{array}
\right.$$
    By definition, $r|_H=Id_H$. It suffices to show that there exists a finite subset $D\subset H$ such that $r(gg')\sim_{D}r(g)r(g')$ for any $g,g'\in G$. Let $D(\rho_i)$ be the defect set of $\rho_i$ for $1\le i\le m$. Hence, the set $A:=\{c_1^{d_1}\cdots c_m^{d_m}: d_1\in D(\rho_1),\cdots, d_m\in D(\rho_m)\}$ is finite in $H$. Also observe that $h_1\sim_Fh_2$ whenever $h_1,h_2\in H$ satisfies $\pi(h_1)=\pi(h_2)$.

    \textbf{Case I: $g,g'\in H\Rightarrow gg'\in H$.} In this case, $r(gg')=gg'=r(g)r(g')$.

    \textbf{Case II: $g\in H, g'\notin H\Rightarrow gg'\notin H$.} In this case, $$r(gg')=c_1^{\rho_1(gg')}\cdots c_m^{\rho_m(gg')}\sim_Ac_1^{\rho_1(g)+\rho_1(g')}\cdots c_m^{\rho_m(g)+\rho_m(g')}\sim_Fc_1^{\rho_1(g)}\cdots c_m^{\rho_m(g)}r(g')\sim_Fr(g)r(g').$$

    \textbf{Case III: $g\notin H, g'\in H\Rightarrow gg'\notin H$.} In this case, $$r(gg')=c_1^{\rho_1(gg')}\cdots c_m^{\rho_m(gg')}\sim_Ac_1^{\rho_1(g)+\rho_1(g')}\cdots c_m^{\rho_m(g)+\rho_m(g')}\sim_Fr(g)c_1^{\rho_1(g')}\cdots c_m^{\rho_m(g')}\sim_Fr(g)r(g').$$

    \textbf{Case IV: $g,g'\notin H$.} If $gg'\in H$, then $$r(gg')=gg'\sim_{F}c_1^{\rho_1(gg')}\cdots c_m^{\rho_m(gg')}\sim_A c_1^{\rho_1(g)+\rho_1(g')}\cdots c_m^{\rho_m(g)+\rho_m(g')}\sim_Fr(g)r(g').$$
    If $gg'\notin H$, then $$r(gg')=c_1^{\rho_1(gg')}\cdots c_m^{\rho_m(gg')}\sim_A c_1^{\rho_1(g)+\rho_1(g')}\cdots c_m^{\rho_m(g)+\rho_m(g')}\sim_Fr(g)r(g').$$

    In summary, we conclude that $r$ is a quasi-homomorphism and the conclusion follows.
\end{proof}

By combining Proposition \ref{Prop: HESNotQR} and Proposition \ref{Prop: HypEmbSubpgQR}, we complete the proof of Theorem \ref{IntroThm2}.

\section{Induced group actions on hyperbolic spaces}\label{sec: IndQA}

\subsection{Preliminaries of group actions on hyperbolic spaces}\label{subsec: HypSpace}

For an isometric action of a group $G$ on a hyperbolic space $X$, we denote by $\partial X$ the \textit{Gromov boundary} of $X$ and $\Lambda (G)$ the \textit{limit set} of $G$ in $\partial X$. 

By Gromov \cite{Gro87}, the isometries of a hyperbolic space $X$ can be subdivided into three classes. A nontrivial element $g\in \Isom(X)$ is called \textit{elliptic} if some $\langle g\rangle$-orbit is bounded. Otherwise, it is called \textit{hyperbolic} (resp. \textit{parabolic}) if it has exactly two fixed points (resp. one fixed point) in the Gromov boundary $\partial X$ of $X$.

\begin{theorem}\cite[Theorem 4.2]{ABO19}\label{Thm: Classification}
    Let $G$ be a group acting isometrically on a hyperbolic space $X$. Then exactly one of the following conditions holds. 
    \begin{enumerate}
        \item $|\Lambda(G)|=0$. Equivalently, $G$ has bounded orbits. In this case the action of $G$ is called {\rm \textbf{elliptic}}.
        \item $|\Lambda(G)|=1$. Equivalently, $G$ has unbounded orbits and contains no hyperbolic elements. In this case the action of $G$ is called {\rm parabolic} or {\rm \textbf{horocyclic}}. A parabolic action cannot be cobounded and the set of points of $\partial X$ fixed by $G$ coincides with $\Lambda(G)$.
        \item $|\Lambda(G)|=2$. Equivalently, $G$ contains a hyperbolic element and any two hyperbolic elements have the same fixed points in $\partial X$. In this case the action of $G$ is called {\rm \textbf{lineal}}.
        \item $|\Lambda(G)|=\infty$. Then $G$ always contains hyperbolic elements. In turn, this case breaks into two subcases.\begin{enumerate}
            \item $G$ fixes a point of $\partial X$. Equivalently, any two hyperbolic elements of $G$ have a common fixed point in the boundary. In this case the action of $G$ is called {\rm quasi-parabolic} or {\rm \textbf{focal}}. Orbits of quasi-parabolic actions are always quasi-convex.
            \item $G$ does not fix any point of $\partial X$. Equivalently, $G$ contains infinitely many independent hyperbolic elements. In this case the action of $G$ is said to be {\rm \textbf{of}\ \textbf{general}\ \textbf{type}}.
        \end{enumerate}
    \end{enumerate}
\end{theorem}

A \textit{quasi-line} is a  metric space quasi-isometric to $\R$ equipped with the standard metric. We list some well-known facts for future use.
\begin{fact}\label{Fact: HypGeo}
    \begin{enumerate}
        \item\label{fac1} Any finitely generated group admitting a focal or general type action on a hyperbolic space  contains a non-abelian free sub-semigroup and thus has exponential growth.
        \item\label{fac2} Group actions on quasi-lines are either elliptic or lineal.
        \item\label{fac3} If a group admits a cobounded lineal action on a hyperbolic space $X$, then $X$ is a quasi-line.
        \item\label{fac4} Let $G=A\times B$ for some groups $A$ and $B$. Suppose that $G$ acts coboundedly on a hyperbolic space $X$. Then the action of $A$ on $X$ is either elliptic or cobounded \cite[Lemma 4.20]{ABO19}.  
        \item\label{fac5} Let $G$ be a group acting isometrically on a hyperbolic space $X$. Let $o\in X$ be a basepoint. Then an element $f\in G$ is a hyperbolic element if and only if the stable length $\tau(f):=\lim_{n\to \infty}\frac{d(o,f^no)}{n}$ of $f$ is strictly bigger than 0. 
        \item\label{fac6} Let $G$ be a group acting isometrically on a hyperbolic space $X$. Let $f\in G$ be a hyperbolic element. Then the subgroup $E(f):=\{g\in G: \exists \ n\ge 1, \text{ s.t. } gf^ng^{-1}=f^n\}$ admits a  lineal action on $X$ with $\Lambda(E(f))=\Lambda\langle f\rangle$. 
    \end{enumerate}
\end{fact}

\paragraph{\textbf{Hyperbolic structures}}
%In this subsection, we recall some basic material about hyperbolic structures on groups and give an equivalent definition of property (PH). For more details about hyperbolic structures on groups, see \cite{ABO19}.
For a group $G$ with a generating set $S$, we denote by $\G(G,S)$ the Cayley graph of $G$ with respect to $S$. 
\begin{defn}\cite[Definition 1.1]{ABO19}
    Let $S$ and $T$ be two generating sets of a group $G$. We say that $S$ is \textit{dominated} by $T$, written $S \preceq T$, if the identity map on $G$ induces a Lipschitz map between metric spaces $\G(G,T)\to \G(G,S)$. It is clear that $\preceq$ is a preorder on the set of generating sets of $G$ and therefore it induces an equivalence relation in the standard way:
    $$S\text{ is equivalent to } T\Longleftrightarrow S\preceq T \text{ and } T\preceq S.$$
\end{defn}

We denote by $[S]$ the equivalence class of a generating set $S$ of $G$. The preorder $\preceq$ induces an order relation $\preccurlyeq$ on all equivalence classes of generating sets by the rule: 
\begin{equation}\label{Equ: Preorder}
    [S]\preccurlyeq [T] \Longleftrightarrow S\preceq T.
\end{equation}
\begin{defn}\cite[Definition 1.2]{ABO19}
    A \textit{hyperbolic structure} on $G$ is an equivalence class $[S]$ such that $\G(G,S)$ is hyperbolic. We denote the set of hyperbolic structures by $\H(G)$ and endow it with the order induced from $(\ref{Equ: Preorder})$.
\end{defn}

Since hyperbolicity of a geodesic metric space is a quasi-isometry invariant, the definition above is independent of the choice of a particular representative in the equivalence class $[S]$. 

Using the standard argument from the proof of the Svarc-Milnor lemma, it is easy to show that elements of $\H(G)$ are in one-to-one correspondence with equivalence classes of cobounded actions of $G$ on hyperbolic spaces considered up to a natural equivalence: two actions $G\curvearrowright X$ and $G\curvearrowright Y$ are \textit{equivalent} if there is a coarsely $G$-equivariant quasi-isometry $X\to Y$. This gives another equivalent definition of property (PH').

\begin{defn}[Second definition of property (PH')]\label{Def: PH}
    A (not necessarily finitely generated) group $G$ has property (PH') if there exist finitely many hyperbolic structures $[S_1],\cdots, [S_n]\in \H(G)$ such that the diagonal action of $G$ on $\prod_{i=1}^n\G(G,S_i)$ equipped with $\ell^1$-metric is proper. 
\end{defn}

\subsection{Inherited cobounded actions on hyperbolic spaces}\label{subsec: InhCobddAct}
In this subsection, we are going to show that cobounded group actions on hyperbolic spaces are inherited by normal quasi-retracts. At first, we need two simple lemmas. 
\begin{lemma}\label{Lem: C(H,T)^n}
    Let $G$ be a group and $H$ be a normal subgroup of $G$. Then for any subset $T\subset G$ and $n\ge 1$,  $\C(H,T^n)\subset \C(H,T)^n$.
\end{lemma}
\begin{proof}
    For any $h\in H, g\in T^{n-1}, t\in T$, observe that $$[h,tg]=htgh^{-1}g^{-1}t^{-1}=h(gh^{-1}g^{-1}\cdot ghg^{-1})tgh^{-1}g^{-1}t^{-1}=[h,g][ghg^{-1},t].$$ Then the conclusion follows from the above equality and mathematical inductions on $n$.
\end{proof}

\begin{lemma}\label{Lem: (ht)^n}
    Let $G$ be a group and $H$ be a normal subgroup of $G$. Let $T\subset G$ be a subset almost commuting with $H$. Then for any $h\in H, t\in T$ and $n\ge 1$,  $(ht)^n\sim_{\C(H,T)^n}h^nt^n$.
\end{lemma}
\begin{proof}
    By Lemma \ref{Lem: SymCommutatorSet}, $\C(H,T)=\C(T^{-1},H)$. It suffices  to show that $a_n=t^{-n}h^{-n}(ht)^n\in \C(T^{-1},H)^n$.

    For each $1\le k\le n$, we denote $a_{n,k}:=t^{-(n-k)}h^{-(n-1)}(th)^{n-k}$. Since $H\lhd G$, it is not hard to see that $a_{n,k}\in H$ for each $1\le k\le n$. Observe that $$a_{n,k}=t^{-(n-k)}h^{-(n-1)}(th)^{n-k}=t^{-1}\cdot t^{-(n-k-1)}h^{-(n-1)}(th)^{n-k-1}\cdot t \cdot h=[t^{-1},a_{n,k+1}]a_{n,k+1}h$$ for each $1\le k\le n-1$ and $$a_{n,n-1}h^{n-2}=t^{-1}h^{-(n-1)}th^{n-1}=[t^{-1}, h^{-(n-1)}]=[t^{-1},a_{n,n}].$$ Therefore, 
    \begin{align*}
        a_n&=t^{-n}h^{-n}(ht)^n=t^{-n}h^{-(n-1)}(th)^{n-1}t\\ &=t^{-1}\cdot t^{-(n-1)}h^{-(n-1)}(th)^{n-1}\cdot t=[t^{-1},a_{n,1}]a_{n,1}\\ &=[t^{-1},a_{n,1}][t^{-1},a_{n,2}]a_{n,2}h\\ &=[t^{-1},a_{n,1}]\cdots [t^{-1},a_{n,n-1}]a_{n,n-1}h^{n-2}\\&=[t^{-1},a_{n,1}]\cdots [t^{-1},a_{n,n}]\in \C(T^{-1},H)^n
    \end{align*}
    and the conclusion follows.
\end{proof}

Let $G$ be a group and $H$ be a normal retract of $G$. Suppose $G$ admits a cobounded action on a hyperbolic space $X$. By Theorem \ref{Lem: ClaVer2}, $G$ is isomorphic to $H\times G/H$. Then by Fact \ref{Fact: HypGeo} (\ref{fac4}), the action of $H$ on $X$ is either elliptic or cobounded. This implies that cobounded actions on hyperbolic spaces are inherited by normal retracts. Next, we are going to show the same conclusion holds for normal quasi-retracts.

\begin{lemma}\label{Lem: SubgpCobdd}
    Let $(G,H)$ be a group pair and $T$ be a right coset transversal of $H$ in $G$ which almost commutes with $H$. Suppose $G$ admits a cobounded action on a hyperbolic space $X$. If there exists a hyperbolic element in $H\cup T$, then the action $H\curvearrowright X$ is either elliptic or cobounded.
\end{lemma}
\begin{proof}
    For each hyperbolic element $f\in G$, we denote $E(f):=\{g\in G: \exists \ n\ge 1, \text{ s.t. } gf^ng^{-1}=f^n\}$. By assumption, we can pick a hyperbolic element $f\in H\cup T$.
    
     \textbf{Case I: $f\in T$.} 
     
     Since $\C(H,T)$ is finite, one gets that $H\le E(f)$. Since the action of $E(f)$ on $X$ is lineal, the action $H\curvearrowright X$ is either elliptic or lineal. If it is elliptic, then we are done. Now suppose the action $H\curvearrowright X$ is lineal. Let $h\in H$ be a hyperbolic element. Then $\Lambda(E(h))=\Lambda(H)=\Lambda(E(f))$.  Similarly, since $\C(H,T)$ is finite, one gets that $T\subset E(h)$. Since $G=HT$, we conclude that $\Lambda (G)=\Lambda(H)$. This implies that the action of $G$ on $X$ is lineal. Hence, $X$ is a quasi-line and $H$ acts coboundedly on $X$.

    \textbf{Case II: $f\in H$.} 
    
    Since $\C(H,T)$ is finite, one gets that $T\subset E(f)$. In particular, $\langle T\rangle\le E(f)$. Since the action of $E(f)$ on $X$ is lineal, the action $\langle T\rangle\curvearrowright X$ is either elliptic or lineal. If it is elliptic, then the orbit of $G$ is within a finite Hausdorff distance to the orbit of $H$. Thus $H\curvearrowright X$ is cobounded since $G\curvearrowright X$ is cobounded. Now suppose the action $\langle T\rangle\curvearrowright X$ is lineal. Let $t\in \langle T\rangle$ be a hyperbolic element. Then $\Lambda(E(t))=\Lambda\langle T\rangle=\Lambda(E(f))$.   Since $\C(H,T)$ is finite, it follows from Lemma \ref{Lem: C(H,T)^n} that $\C(H,t)$ is finite. Similarly, one gets that $H\le  E(t)$. Since $G=HT=H\langle T\rangle$, we conclude that $\Lambda (G)=\Lambda(E(t))$. This implies that the action of $G$ on $X$ is lineal.  Hence, $X$ is a quasi-line and $H$ acts coboundedly on $X$.

    In summary, we complete the proof.
\end{proof}

\begin{lemma}\label{Lem: NorQRInhCobddAct}
    Let $G$ be a group and $H$ be a normal quasi-retract of $G$. Suppose $G$ admits a cobounded action on a hyperbolic space $X$. Then the action $H\curvearrowright X$ is either elliptic or cobounded.
\end{lemma}
\begin{proof}
    By Theorem \ref{Thm: LQS}, there exists a section $s: G/H\to G$ such that $s$ is a quasi-homomorphism and $\C(H,s(G/H))$ is finite. Denote $T=s(G/H)$ for simplicity. Note that any element $g\in G$ can be uniquely written as $g=g_Hg_T$ where $g_H\in H$ and $g_T\in T$.

    Since the action $G\curvearrowright X$ is cobounded, it belongs to one of the following types: elliptic, lineal, focal and of general type. If the action $G\curvearrowright X$ is elliptic, then $X$ is a bounded metric space and the conclusion holds obviously. If the action $G\curvearrowright X$ is lineal, then $X$ is a quasi-line and the conclusion follows from Fact \ref{Fact: HypGeo} (\ref{fac2}). Now suppose that the action $G\curvearrowright X$ is either focal or of general type.

    Fix a basepoint $o\in X$. Denote $M:=\max_{g\in \C(H,T)}d(o,go)$. Pick a hyperbolic element $f\in G$ such that its stable length satisfies $\tau(f)>M$. By the above arguments, $f=f_Hf_T$ where $f_H\in H$ and $f_T\in T$. By Lemma \ref{Lem: (ht)^n}, for any $n\ge 1$, there exist $a_1,\cdots, a_n\in \C(H,T)$ such that $f^n=(f_Hf_T)^n=f_H^nf_T^na_1\cdots a_n$. By the triangle inequality, $$d(o,f^no)=d(o,f_H^nf_T^na_1\cdots a_no)\le d(o,f_H^no)+d(o,f_T^no)+\sum_{i=1}^nd(o,a_io)\le d(o,f_H^no)+d(o,f_T^no)+nM.$$ By dividing by $n$ on both sides and letting $n\to \infty$, one gets that $0<\tau(f)-M\le \tau(f_H)+\tau(f_T)$. This implies that at least one of $f_H$ and $f_T$ is hyperbolic.  The remaining proof is completed by Lemma \ref{Lem: SubgpCobdd}.
\end{proof}
\begin{remark}
    The conclusion of Lemma \ref{Lem: NorQRInhCobddAct} does not hold for general quasi-retracts. For example, let $G=\F_2$ with a free basis $\{a,b\}$ and $H=\langle b\rangle$. Clearly $H$ is a retract of $G$. Define a homomorphism $\rho: G\to \sl_2(\R)$ by mapping 
    $$a=\left(
  \begin{array}{cc}
    \sqrt{2} & 0 \\
    0 & 1/\sqrt{2} \\
  \end{array}
\right) \text{ and } b=\left(
  \begin{array}{cc}
    1 & 1 \\
    0 & 1 \\
  \end{array}
\right).$$
Then we obtain an action of $G$ on $\mathbb H^2$. This action is focal and cobounded (cf. \cite[Proposition 4.5]{TW25}). However, the orbit of $H$ in $\mathbb H^2$ is not even quasi-convex.
\end{remark}

As a corollary, we have
\begin{corollary}\label{Cor: NorQRHasPH'}
    Property (PH') is inherited by normal quasi-retracts.
\end{corollary}

For an independent interest, we give a classification about the action of the defect subgroup with respect to a quasi-homomorphism on a hyperbolic space.

\begin{lemma}\label{Lem: GeoOfDelta}
    Let $\phi: G\to H$ be a surjective quasi-homomorphism with $D=D(\phi)$. Suppose $H$ admits a cobounded action on a hyperbolic space $X$. Then the action of the defect subgroup $\Delta_{\phi}$ on $X$ is either elliptic or lineal. Moreover, one of the following holds: 
    \begin{enumerate}
        \item if $\Delta_{\phi}\curvearrowright X$ is lineal, then $\Delta_{\phi}\curvearrowright X$ is cobounded and $X$ is a quasi-line.
        \item if $\Delta_{\phi}\curvearrowright X$ is elliptic, then the quotient space $X/\Delta_{\phi}$ endowed with the quotient metric is quasi-isometric to $X$ and thus is still a hyperbolic space.
    \end{enumerate} 
\end{lemma}
\begin{proof}
    By Lemma \ref{Lem: SujQHImpEle}, $\Delta_{\phi}=\langle D\rangle$ is a virtually abelian normal subgroup of $H$. Since any finitely generated group admitting a focal or general type action on a hyperbolic space has exponential growth, the action of $\Delta_{\phi}$ on $X$ is either elliptic, or horocyclic, or lineal. 

    \begin{claim}
        The action $\Delta_{\phi}\curvearrowright X$ is not horocyclic.
    \end{claim}
    \begin{proof}[Proof of Claim]
        Suppose to the contrary that the action $\Delta_{\phi}\curvearrowright X$ is horocyclic. Then every element of $\Delta_{\phi}$ is either elliptic or parabolic on $X$. A well-known fact is that if a finitely generated virtually abelian group acts isometrically on a hyperbolic space with each element being elliptic, then the action is also elliptic. Hence, there exists at least one parabolic element, say $f$, in $\Delta_{\phi}$ on $X$. Since $\Delta_{\phi}$ is finitely generated by $D$, there exists $N\ge 1$ such that $f\in D^N\cup D^{-N}$. By Lemma \ref{Lem: QHProperties} (\ref{conjugate}), for any $h\in H$ and $n\ge 1$, $h^nfh^{-n}\in A:=(D^2D^{-1})^N\cup (D^2D^{-1})^{-N}$. Fix a basepoint $o\in X$ and denote $M:=\max_{g\in A}d(o,go)$.

        Since $\Delta_{\phi}\lhd H$ and $H\curvearrowright X$ is cobounded, the action of $H$ on $X$ must be focal (cf. \cite[Lemma 4.21]{ABO19}). Then there exists a hyperbolic element $h\in H$ such that $d(o,h^nfh^{-n}o)\to \infty$ as $n\to \infty$. This contradicts with $d(o,h^nfh^{-n}o)\le M$ for any $n\ge 1$. Therefore, the action $\Delta_{\phi}\curvearrowright X$ is not horocyclic.
    \end{proof}

    By the above Claim, the action of $\Delta_{\phi}$ on $X$ is either elliptic or lineal. 
    
    If $\Delta_{\phi}\curvearrowright X$ is lineal, then it follows from $\Delta_{\phi}\lhd H$  that the action $H\curvearrowright X$ is also lineal. Since $H\curvearrowright X$ is also cobounded, one gets that $X$ is a quasi-line and thus $\Delta_{\phi}\curvearrowright X$ is also cobounded.

    If $\Delta_{\phi}\curvearrowright X$ is elliptic, then the conclusion can be deduced from \cite[Lemma 4.10]{BFFG22}.
\end{proof}

\subsection{Induced quasi-actions}\label{subsec: IndQA}
\begin{defn}\cite[Definition 2.3]{Man06}\label{Def: QuasiAction}
    A \textit{$(\lambda, \epsilon)$-quasi-action} of a group $G$ on a metric space $X$ is a map $\rho: G \times X \to X$, denoted $\rho(g, x) \mapsto gx$, so that the following hold:
    \begin{itemize}
        \item[(i)] for each $g$, $\rho(g, -): X \to X$ is a $(\lambda, \epsilon)$-quasi-isometry;
        \item[(ii)] for each $x \in X$ and $g, h \in G$, we have $$d(g(hx),(gh)x) = d(\rho(g, \rho(h, x)), \rho(gh, x)) \le \epsilon.$$
    \end{itemize}
(Note that $\lambda$ and $\epsilon$ must be independent of $g$ and $h$.) We call a quasi-action
\textit{cobounded} if, for every $x \in X$, the map $\rho(-, x): G \to X$ is $\epsilon'$-coarsely onto for some $\epsilon'\ge 0$. 
\end{defn}

%When $X$ is a hyperbolic space on which $G$ quasi-acts, we say $g\in G$ quasi-acts \textit{elliptically} (resp. \textit{hyperbolically}) on $X$ if the orbit of $\langle g\rangle$ in $X$ is bounded (resp. a quasi-geodesic) 

\begin{defn}\cite[Definition 2.4]{Man06}
    Suppose that $\rho_X : G \times X \to X$ and $\rho_Y : G \times Y \to Y$ are quasi-actions. A map $\sigma : X \to Y$ is called \textit{coarsely equivariant} if there is some $\epsilon\ge 0$ so that $d(\sigma \circ \rho_X (g, x), \rho_Y (g, \sigma(x))) \le \epsilon$ for all $g \in G$ and $x \in X$. A coarsely equivariant quasi-isometry is called a \textit{quasi-conjugacy}. If there exists a quasi-conjugacy between $X$ and $Y$, then the two quasi-actions are called \textit{quasi-conjugate}.
\end{defn}

By definition, two equivalent cobounded actions are quasi-conjugate.

\begin{example}\label{Ex: QuasiAction}
    Let $\phi : G \to \R$ be a quasimorphism with $\|D(\phi)\|_{\infty}=\epsilon$. A $(1, \epsilon)$-quasi-action of $G$ on $\R$ is given by $\rho(g, x) = x + \phi(g)$.
\end{example}

Let $\phi: G\to H$ be a quasi-homomorphism with $D=D(\phi)$. Suppose $H$ admits a $(\lambda,\epsilon)$-quasi-action on a metric space $X$ and there exists a constant $M>0$ such that $\sup_{x\in X, a\in D}d(x,ax)\le M$. Define a map $\rho: G\times X\to X$ by $\rho(g,x):=\phi(g)x$. For any $g,h\in G$, there exists $a\in D$ such that $\phi(gh)=\phi(g)\phi(h)a$. Thus, for each $x\in X$ and $g,h\in G$, we have 
\begin{align*}
    d(\rho(g,\rho(h,x)),\rho(gh,x))&=d(\phi(g)(\phi(h)x), \phi(gh)x)\le d((\phi(g)\phi(h))x,(\phi(g)\phi(h)a)x)+\epsilon\\ &\le d((\phi(g)\phi(h))x,(\phi(g)\phi(h))(ax))+2\epsilon\le \lambda d(x,ax)+3\epsilon\le \lambda M+3\epsilon.
\end{align*}
The above inequality shows that $\rho$ is a $(\lambda,\lambda M+3\epsilon)$-quasi-action of $G$ on $X$. This quasi-action is referred to as \textit{an induced quasi-action by $\phi$}. Furthermore, if $\phi$ is a quasi-retraction, then the induced quasi-action of $G$ on $X$ is an extension of the original quasi-action of $H$ on $X$. 

\begin{lemma}\label{Lem: QHToQA}
    Let $\phi: G\to H$ be a coarsely surjective quasi-homomorphism. Suppose $H$ admits a cobounded quasi-action on a metric space $X$. Then the map $\rho: G\times X\to X$ defined by $\rho(g,x):=\phi(g)x$ is an induced quasi-action of $G$ on $X$ by $\phi$. 
    
    Moreover, suppose $H$ admits another cobounded quasi-action on a metric space $Y$. Then the quasi-action of $H$ on $X$ is quasi-conjugate to the quasi-action of $H$ on $Y$ if and only if the induced quasi-action of $G$ on $X$ by $\phi$ is quasi-conjugate to the induced quasi-action of $G$ on $Y$ by $\phi$.
\end{lemma}
\begin{proof}
    Let $D=D(\phi)$ be the defect set of $\phi$. By the above arguments, it suffices to show that there exists a constant $M>0$ such that $\sup_{x\in X, a\in D}d(x,ax)\le M$.

    Since $\phi$ is coarsely surjective, there exists a constant $C>0$ such that $H=\mathcal N_C(\phi(G))$. Since $H$ is discrete, this is equivalent to saying that there exists a finite subset $A\subset H$ such that $H=\phi(G)\cdot A$. Fix a basepoint $o\in X$. Denote $M_1:=\max_{b\in A^{-1}D^2D^{-1}A}d(o,bo)$. Since the quasi-action of $H$ on $X$ is cobounded, there exists a constant $M_2>0$ such that for any $x\in X$, there exists $h_x\in H$ such that $d(x,h_xo)\le M_2$. Since $H=\phi(G)A$, there exists $g_x\in G$ and $a_x\in A$ such that $h_x=\phi(g_x)a_x$. Let $\lambda\ge 1, \epsilon\ge 0$ be two constants such that $H$ admits a $(\lambda,\epsilon)$-quasi-action on $X$. By the triangle inequality and Lemma \ref{Lem: QHProperties} (\ref{conjugate}), we have that 
    \begin{align*}
        d(x,ax)&\le d(x,h_xo)+d(h_xo,a(h_xo))+d(a(h_xo),ax)\le d(h_xo,(ah_x)o)+(\lambda+1) M_2+2\epsilon\\ &\le \lambda d(o,h_x^{-1}(ah_xo))+(\lambda+1) M_2+3\epsilon\le \lambda d(o,(h_x^{-1}ah_x)o)+(\lambda+1) M_2+4\epsilon\\ &= \lambda d(o,(a_x^{-1}\phi(g_x)^{-1}a\phi(g_x)a_x)o)+(\lambda+1) M_2+4\epsilon \le \lambda M_1+(\lambda+1) M_2+4\epsilon.
    \end{align*}
    By setting $M=\lambda M_1+(\lambda+1) M_2+4\epsilon$, one gets that $\rho$ is a $(\lambda,\lambda M+3\epsilon)$-quasi-action of $G$ on $X$.

    Moreover, suppose that there exists a coarsely $H$-equivariant quasi-isometry $\sigma: X\to Y$. By the above constructions, this map $\sigma$ gives also an induced coarsely $G$-equivariant quasi-isometry $\sigma: X\to Y$. And vice versa since $\phi$ is coarsely surjective.
\end{proof}

\begin{defn}\cite[Definition 1.2]{Man06}
    A group $G$ has \textit{property (QFA)} if for every quasi-action of $G$ on any tree $X$, there is some $x\in X$ so that the orbit $Gx$ has finite diameter.
\end{defn}

\begin{lemma}
    Let $G,H$ be two finitely generated groups and $\phi: G\to H$ be a coarsely surjective quasi-homomorphism. If $G$ has property (QFA), then so is $H$.
\end{lemma}
\begin{proof}
    Let $X$ be an arbitrary tree on which $H$ quasi-acts. By \cite[Remark 3.2]{Man06}, the quasi-action of $H$ on $X$ can be assumed to be cobounded. As a result of Lemma \ref{Lem: QHToQA}, $\phi$ induces a cobounded quasi-action of $G$ on $X$. Since $G$ has property (QFA), there is some $x\in X$ so that the orbit $Gx:=\phi(G)x$ has finite diameter. Since $\phi$ is coarsely surjective, the orbit $Hx$ has also finite diameter. Therefore, $H$ has property (QFA) since $X$ is arbitrary.
\end{proof}

\begin{corollary}\label{Cor: QFAIsQIInv}
    Property (QFA) is a quasi-isomorphism invariant for finitely generated groups.
\end{corollary}

\begin{proposition}\label{Prop: QFAStability}
    Let $1\to H\to G\to Q\to 1$ be a short exact sequence of finitely generated groups which is left quasi-split. Then $G$ has property (QFA) if and only if  both $H$ and $Q$ have property (QFA).
\end{proposition}
\begin{proof}
    By Theorem \ref{Thm: LQS}, $G$ is strictly quasi-isomorphic to $H\times Q$. As a result of Corollary \ref{Cor: QFAIsQIInv}, ``$G$ has property (QFA)'' $\Leftrightarrow$  ``$H\times Q$ has property (QFA)'' $\Leftrightarrow$ ``both $H$ and $Q$ have property (QFA)''.
\end{proof}

To obtain the second item of Theorem \ref{IntroThm: PH&QTStability}, we need the following fact.

\begin{proposition}\cite[Proposition 4.4]{Man05}\label{Prop: QAToA}
    Any quasi-action on a geodesic metric space $X$ is quasi-conjugate to an action on some connected graph quasi-isometric to $X$. Conversely, any isometric action on a geodesic metric space quasi-isometric to $X$ is quasi-conjugate to some quasi-action on $X$.
\end{proposition}

\begin{lemma}\label{Lem: PH'IsQIInv}
    Property (PH') is a quasi-isomorphism invariant for discrete groups.
\end{lemma}
\begin{proof}
    Let $G,G'$ be two discrete quasi-isomorphic groups. Suppose that $G'$ has property (PH') from actions $G'\curvearrowright X_1,\cdots, G'\curvearrowright X_n$. Let $\phi: G\to G'$ be a quasi-isomorphism. By Lemma \ref{Lem: QHToQA}, $\phi$ induces a cobounded quasi-action of $G$ on each $X_i$. As a result of Proposition \ref{Prop: QAToA}, for each $1\le i\le n$, the induced quasi-action of $G$ on $X_i$ is quasi-conjugate to an action of $G$ on a hyperbolic space $X_i'$. Via these quasi-conjugacy, we know that each action $G\curvearrowright X_i'$ is also cobounded. Moreover, since $\phi$ is a quasi-isomorphism, the (PH') property of $G'$ implies that the diagonal action of $G$ on the $\ell^1$-product space $\prod_{i=1}^nX_i'$ is proper. This gives the (PH') property of $G$.
\end{proof}

\begin{lemma}\cite[Lemma 7.1]{TW25}\label{Lem: DirPro}
    Let $G$ be a group splitting as a direct product $G=H\times K$. Then 
    \begin{enumerate}
        \item\label{6.4.1} $G$ has property (PH') if and only if both of $H,K$ have property (PH').
        \item\label{6.4.2} $G$ has property (PH) if and only if both of $H,K$ have property (PH).
        \item\label{6.4.3} assume $G$ is finitely generated, $G$ has property (QT') or (QT) if and only if both of $H,K$ have property (QT') or (QT), respectively.
    \end{enumerate}    
\end{lemma}

\begin{proposition}\label{Prop: PH'Stability}
    Let $1\to H\to G\to Q\to 1$ be a short exact sequence of discrete groups which is left quasi-split. Then $G$ has property (PH') if and only if  both $H$ and $Q$ have property (PH').
\end{proposition}
\begin{proof}
    By Theorem \ref{Thm: LQS}, $G$ is strictly quasi-isomorphic to $H\times Q$. As a result of Lemma \ref{Lem: PH'IsQIInv}, $G$ has property (PH') if and only if  $H\times Q$ has property (PH'). Then the conclusion follows from Lemma \ref{Lem: DirPro} (\ref{6.4.1}).
\end{proof}

%Recall that a finitely generated group $G$ has property (QT) if $G$ acts isometrically on a finite product of quasi-trees equipped with $\ell^1$-metric such that the orbit map is a quasi-isometric embedding. Following \cite{Tao24}, a finitely generated group $G$ has \textit{property (QT')} if it has property (QT) and the action on a finite product of quasi-trees is induced by a diagonal action. By a result of Button \cite{But22}, a finitely generated group with property (QT) virtually has property (QT'). 

As analogues of Lemma \ref{Lem: PH'IsQIInv} and Proposition \ref{Prop: PH'Stability}, we have

\begin{lemma}\label{Lem: QT'IsQIInv}
    Property (QT') is a quasi-isomorphism invariant for finitely generated groups.
\end{lemma}
\begin{proof}
    By \cite[Remark 3.2]{Man06}, any isometric group action of a finitely generated group on a quasi-tree can be assumed to be cobounded. The remaining proof is similar to the proof of Lemma \ref{Lem: PH'IsQIInv}.
\end{proof}

In \cite{But22}, Button showed that property (QT) is not a quasi-isometry invariant. Since quasi-isomorphisms are stronger equivalence relations than quasi-isometries for finitely generated groups, we can ask the following question.
\begin{question}
    Is property (QT) a quasi-isomorphism invariant for finitely generated groups?
\end{question}

By combining Lemmas \ref{Lem: AC}, \ref{Lem: PH'IsQIInv}, \ref{Lem: QT'IsQIInv} and Corollary \ref{Cor: QFAIsQIInv}, we obtain Proposition \ref{IntroProp: QIInv}.

\begin{proposition}\label{Prop: QT'Stability}
    Let $1\to H\to G\to Q\to 1$ be a short exact sequence of finitely generated groups which is left quasi-split. Then $G$ has property (QT') if and only if  both $H$ and $Q$ have property (QT').
\end{proposition}
\begin{proof}
    By Theorem \ref{Thm: LQS}, $G$ is strictly quasi-isomorphic to $H\times Q$. As a result of Lemma \ref{Lem: QT'IsQIInv}, $G$ has property (QT') if and only if  $H\times Q$ has property (QT'). Then the conclusion follows from Lemma \ref{Lem: DirPro} (\ref{6.4.3}).
\end{proof}

To obtain the third item of Theorem \ref{IntroThm: PH&QTStability}, we need the following intermediate result.

In \cite{NW22}, Nguyen-Wang gave a geometric characterization for FZ groups via a notion of curvature for finitely generated groups. Here, we collect and give various equivalent algebraic characterizations for FZ groups, one of which is via quasi-retracts.

\begin{proposition}\label{Prop: ElementaryGp}
    Let $G$ be a finitely generated group. The following are equivalent:
    \begin{enumerate}
        \item\label{i1} $G$ is FZ, i.e. $[G:Z(G)]<\infty$.
        \item\label{i2} $[G,G]$ is finite.
        \item\label{i3} $\C(G,G)$ is finite.
        \item\label{i4} $G$ is FC, i.e. each $g\in G$ has finitely many conjugates.
        \item\label{i5} $G$ is finite-by-$\Z^m$.
        \item\label{i6} $G$ is finite-by-abelian.
        \item\label{i7} $G$ is virtually solvable and each of its subgroups is a quasi-retract.
        \item\label{i8} $G$ has property (PH') from orientable lineal structures, i.e. there exist finitely many quasi-lines $L_1,\cdots,L_n$ on which $G$ acts orientable lineally such that the diagonal action of $G$ on $\prod_{i=1}^nL_i$ equipped with $\ell^1$-metric is proper.
    \end{enumerate}
\end{proposition}
\begin{proof}
    $(\ref{i1})\Leftrightarrow (\ref{i2})$ is given by \cite[Theorem 5.3, Corollary 5.41]{Neu51}. $(\ref{i2})\Rightarrow (\ref{i3})\Rightarrow (\ref{i4})\Rightarrow (\ref{i5})\Rightarrow (\ref{i6})\Rightarrow (\ref{i2})$ are all obvious except that $(\ref{i4})\Rightarrow (\ref{i5})$ is given by \cite[Theorem 5.1]{Neu51}. $(\ref{i5})\Rightarrow (\ref{i8})$ is obvious and $(\ref{i8})\Rightarrow (\ref{i1})$ follows from \cite[Lemma 3.1]{TW25}. It remains to show that $(\ref{i1})\Rightarrow (\ref{i7})$ and $(\ref{i7})\Rightarrow (\ref{i3})$.

    $(\ref{i1})\Rightarrow (\ref{i7})$: Let $Z=Z(G)$ and $H$ be an arbitrary subgroup of $G$. By assumption, $[G:Z]<\infty$. Then $Z$ is a finitely generated abelian group. By Lemma \ref{Lem: ExaOfQR} (\ref{5.4.6}), $H\cap Z$ is a quasi-retract of $Z$. Since $Z$ commutes with $G$, Theorem \ref{Thm: LQS} shows that $Z$ is a quasi-retract of $G$. Thus, as a composition, $H\cap Z$ is a normal quasi-retract of $G$. Note that $H/H\cap Z\cong HZ/Z$. This implies that $[H:H\cap Z]<\infty$. By Lemma \ref{Lem: FinInd}, one gets that $H$ is a quasi-retract of $G$.
    
    $(\ref{i7})\Rightarrow (\ref{i3})$: Since $G$ is virtually solvable, we can choose a finite-index, normal and solvable subgroup $A$ of $G$.  By assumption, $A$ is a quasi-retract of $G$. As a result of Theorem \ref{Thm: LQS}, $G$ is strictly quasi-isomorphic to $A\times G/A$.   

    Denote $A_0=A$ and $A_i=[A_{i-1},A_{i-1}]$ for $i\ge 1$. Since $A$ is solvable, there exists $n\in \N$ such that $A_n=1$. Note that $A_i\lhd A_{i-1}\le G$. By composing the embedding map $A_{i-1}\to G$ with the quasi-retraction $G\to A_i$, one gets that each $A_i$ is a quasi-retract of $A_{i-1}$ for $1\le i\le n$. As a result of Theorem \ref{Thm: LQS}, $A_{i-1}$ is strictly quasi-isomorphic to $A_i\times A_{i-1}/A_i$. Since strict quasi-isomorphism is an equivalent relation on discrete groups (cf. Lemma \ref{Lem: EquiRelation}), one gets that $G$ is strictly quasi-isomorphic to $G/A\times \prod_{i=1}^nA_{i-1}/A_i$ which is almost commutative. By Lemma \ref{Lem: AC}, $\C(G,G)$ is also finite.
\end{proof}
\begin{remark}
    Together with \cite[Lemma 3.1]{TW25}, Proposition \ref{Prop: ElementaryGp} (\ref{i8}) also gives the (QT') property of finitely generated FZ groups.
\end{remark}

As an immediate corollary of Lemma \ref{Lem: AC} and Proposition \ref{Prop: ElementaryGp}, we have

\begin{corollary}
    Being FZ is a quasi-isomorphism invariant for finitely generated groups.
\end{corollary}

\begin{proposition}\label{Prop: PH&QTStability}
    Let $1\to H\to G\xrightarrow{\pi} Q\to 1$ be a short exact sequence of finitely generated groups which is left quasi-split. Suppose that $H$ is FZ. Then $G$ has property (PH) (resp. (QT)) if and only if $Q$ has property (PH) (resp. (QT)).
\end{proposition}
\begin{proof}
    At first, we show that ``$G$ has property (PH) $\Leftrightarrow$ $Q$ has property (PH)''.
    
    ``$\Rightarrow$'': Let $G'$ be a finite-index subgroup of $G$ such that $G'$ has property (PH'). Denote $H'=H\cap G'$ and $Q'=\pi(G')$. By the isomorphism theorem of groups, $G'/(H\cap G')\cong HG'/H=\pi(G')$. Thus, we obtain a new short exact sequence of groups as follows:
    \begin{equation}\label{Equ: FinIndSES}
        1\to H'\to G'\xrightarrow{\pi|_{G'}} Q'\to 1.
    \end{equation}
    $$$$
    
    \begin{claim}
        The short exact sequence (\ref{Equ: FinIndSES}) is left quasi-split.
    \end{claim}
    \begin{proof}[Proof of Claim]
        Since $H$ is FZ, Proposition \ref{Prop: ElementaryGp} shows that $\C(H,H)$ is finite. Note that $[H:H']\le [G:G']<\infty$. As a result of Proposition \ref{Prop: FinInd}, $H'$ is a quasi-retract of $H$. Let $r_1: G\to H$ and $r_2: H\to H'$ be two quasi-retractions. Then the composite map $r=r_2\circ r_1|_{G'}: G'\to H'$ is a quasi-homomorphism and left-inverse to the inclusion map $H'\to G'$. This shows that $H'$ is a quasi-retract of $G'$ and the conclusion follows.
    \end{proof}
    
    By the above Claim, together with Proposition \ref{Prop: PH'Stability}, one gets that $Q'$ has property (PH'). Since $[Q:Q']\le [G:G']<\infty$, $Q$ has property (PH).

    ``$\Leftarrow$'': Let $Q'$ be a finite-index subgroup of $Q$ such that $Q'$ has property (PH'). Denote $G'=\pi^{-1}(Q')$. Thus, we obtain a new short exact sequence of groups as follows:
    $$1\to H\to G'\xrightarrow{\pi|_{G'}} Q'\to 1.$$
    By Proposition \ref{Prop: ElementaryGp}, $H$ has property (PH'). As a result of Proposition \ref{Prop: PH'Stability}, $G'$ has property (PH'). Since $[G:G']=[Q:Q']<\infty$, $G$ has property (PH).

    By substituting Proposition \ref{Prop: QT'Stability} for Proposition \ref{Prop: PH'Stability}, the same arguments show that $G$ has property (QT) if and only if $Q$ has property (QT). 
\end{proof}

By combining Propositions \ref{Prop: QFAStability}, \ref{Prop: PH'Stability}, \ref{Prop: QT'Stability} and \ref{Prop: PH&QTStability}, we obtain Theorem \ref{IntroThm: PH&QTStability}. Note that Proposition \ref{Prop: PH&QTStability} requires an additional assumption on $H$. If the answer to the following question is yes, then we can remove this assumption.

\begin{question}
    Let $G$ be a group and $G'\le G$ be a finite-index subgroup. Suppose that $H$ is a normal quasi-retract of $G$. Is $H\cap G'$ necessarily a normal quasi-retract of $G'$?
\end{question}

\begin{remark}
    If one drops the word ``normal'' in the above question, then the answer is negative. In \cite[Theorem A]{STZ22}, Snopce-Tanushevski-Zalesskii proved that for a free group $\F_n$, there exists a finite-index subgroup $H$ and a retract $R$ of $\F_n$ such that $H\cap R$ is not a retract of $H$. 
\end{remark}

In the end of this subsection, we prove Proposition \ref{IntroProp: HypStr}. The proof will be divided into the next two propositions.

\begin{proposition}\label{Prop: EmbHypStr}
    Let $\phi: G\to H$ be a coarsely surjective quasi-homomorphism. Then $\H(H)$ embeds in $\H(G)$ as a sub-poset.
\end{proposition}
\begin{proof}
    For each $[T]\in \H(H)$, $\G(H,T)$ is a hyperbolic graph on which $H$ acts coboundedly. By Lemma \ref{Lem: QHToQA},  $\phi$ induces a cobounded quasi-action of $G$ on $\G(H,T)$. As a result of Proposition \ref{Prop: QAToA}, the induced cobounded quasi-action of $G$ on $\G(H,T)$ is quasi-conjugate to a cobounded action of $G$ on a hyperbolic graph $X_{T}$. By the classical Milnor-Svarc Lemma, there exists a generating set $S$ of $G$ such that $G\curvearrowright \G(G,S)$ is quasi-conjugate to $G\curvearrowright X_{T}$. Then we can define a map $\chi: \H(H)\to \H(G)$ by $\chi([T])=[S]$. It follows from the ``Moreover'' part of Lemma \ref{Lem: QHToQA} that $\chi$ is well-defined and injective.

    It remains to show that $\chi$ is order-preserving. Let $[T]\preccurlyeq [T']$ be two hyperbolic structures in $\H(H)$. By definition, the identity map induces a Lipschitz map $\sigma: \G(H,T')\to \G(H,T)$. According to Lemma \ref{Lem: QHToQA}, $\sigma$ is also an induced coarsely $G$-equivariant Lipschitz map from $\G(H,T')$ to $\G(H,T)$. After applying Proposition \ref{Prop: QAToA} and the classical Milnor-Svarc Lemma, we get that there exists a coarsely $G$-equivariant coarsely Lipschitz map from $\G(G,S')$ to $\G(G,S)$ where $S\in \chi([T])$ and $S'\in \chi([T'])$. This shows that $\chi([T])\preccurlyeq \chi([T'])$.

    In summary, $\chi$ is an order-preserving injection from $\H(H)$ to $\H(G)$.
\end{proof}

\begin{proposition}\label{Prop: IsoHypStr}
    Let $G,G'$ be two discrete quasi-isomorphic groups. Then $\H(G)\cong \H(G')$.
\end{proposition}
\begin{proof}
    For each $[S']\in \H(G')$, $\G(G',S')$ is a hyperbolic graph on which $G'$ acts coboundedly. Let $\phi: G\to G'$ be a quasi-isomorphism. By Lemma \ref{Lem: QHToQA},  $\phi$ induces a cobounded quasi-action of $G$ on $\G(G',S')$. As a result of Proposition \ref{Prop: QAToA}, the induced cobounded quasi-action of $G$ on $\G(G',S')$ is quasi-conjugate to a cobounded action of $G$ on a hyperbolic graph $X_{S'}$. By the classical Milnor-Svarc Lemma, there exists a generating set $S$ of $G$ such that $G\curvearrowright \G(G,S)$ is quasi-conjugate to $G\curvearrowright X_{S'}$. Then we can define a map $\chi: \H(G')\to \H(G)$ by $\chi([S'])=[S]$. It follows from the ``Moreover'' part of Lemma \ref{Lem: QHToQA} that $\chi$ is well-defined and injective.
    
    Since $\phi: G\to G'$ is a quasi-isomorphism, there exists a quasi-inverse $\phi': G'\to G$ such that $\d(\phi'\circ \phi, Id_{G})<\infty$ and $\d(\phi\circ \phi', Id_{G'})<\infty$. Similarly, we can define a map $\chi': \H(G)\to \H(G')$. It follows from the above definitions that $\chi\circ \chi'([S])=[S]$ for any $[S]\in \H(G)$. This shows that $\chi$ is surjective.

    The property that $\chi$ is order-preserving follows from the same arguments in the proof of Proposition \ref{Prop: EmbHypStr}. In summary, $\chi$ is an order-preserving bijection from $\H(G')$ to $\H(G)$.
\end{proof}

By combining Proposition \ref{Prop: EmbHypStr} and Proposition \ref{Prop: IsoHypStr}, we complete the proof of Proposition \ref{IntroProp: HypStr}.

%\printbibliography
\bibliographystyle{amsplain}   
\bibliography{Reference}
\end{document}